\documentclass[reqno,12pt,a4wide]{article}

\usepackage{eucal,enumerate,mathrsfs}
\usepackage{amsmath,amssymb,epsfig,bbm,amsthm}

\numberwithin{equation}{section}

\newtheorem{theorem}{Theorem}[section]

\newtheorem{lemma}[theorem]{Lemma}
\newtheorem{proposition}[theorem]{Proposition}
\newtheorem{definition}[theorem]{Definition}

\newtheorem{example}[theorem]{Example}
\newtheorem{assumption}[theorem]{Assumption}

\usepackage{hyperref}

\newcommand{\I}{\mathbb{I}}

\newcommand{\N}{\mathbb{N}}

\newcommand{\R}{\mathbb{R}}

\newcommand{\CC}{\mathscr{C}}

\newcommand{\LL}{\mathscr{L}}

\newcommand{\NN}{\mathscr{N}}

\renewcommand{\SS}{\mathscr{S}}

\newcommand{\VV}{\mathscr{V}}
\newcommand{\WW}{\mathscr{W}}

\newcommand{\cA}{{\ensuremath{\mathcal A}}}

\newcommand{\cO}{{\ensuremath{\mathcal O}}}
\newcommand{\cP}{{\ensuremath{\mathcal P}}}

\newcommand{\cY}{{\ensuremath{\mathcal Y}}}

\newcommand{\Kliminf}{K\kern-3pt-\kern-2pt\mathop{\rm
lim\,inf}\limits}  
\newcommand{\Klimsup}{K\kern-3pt-\kern-2pt\mathop{\rm lim\,sup}\limits}  
\renewcommand{\d}{{\mathrm d}}

\newcommand{\restr}[1]{\lower3pt\hbox{$|_{#1}$}}

\newcommand{\weakto}{\rightharpoonup}

\newcommand{\nchi}{{\raise.3ex\hbox{$\chi$}}}
\newcommand{\Rd}{{\R^d}}
\newcommand{\Rn}{{\R^n}}

\newcommand{\RN}{{\R^N}}

\def\qed{\ifmmode 
  \else \leavevmode\unskip\penalty9999 \hbox{}\nobreak\hfill
  \fi               
    \qquad           \hbox{\hskip.5em $\square$
                \hskip.1em}}

\title{$\Gamma$-convergence and relaxations for gradient flows in metric spaces: a minimizing movement approach}
\begin{document}

\author{
Florentine Flei$\ss$ner 
\thanks{Technische Universit\"at M\"unchen  email:
  \textsf{fleissne@ma.tum.de}.
  } 
}

\date{}

\maketitle

\begin{abstract} 
  We present new abstract results on the interrelation between the minimizing movement scheme for gradient flows along a sequence of $\Gamma$-converging functionals and the gradient flow motion for the corresponding limit functional, in a general metric space. We are able to allow a relaxed form of minimization in each step of the scheme, and so we present new relaxation results too.   
\end{abstract}

{\small\tableofcontents}

\section{Introduction}
For a sequence of $\Gamma$-converging functionals $\phi_\epsilon \stackrel{\Gamma}{\to} \phi$, one can consider the minimizing movement scheme for gradient flows along $(\phi_\epsilon)_{\epsilon > 0}$, in which every time step $\tau$ is associated with a parameter $\epsilon = \epsilon(\tau)$, simultaneously passing to the limit in the time steps $\tau\to0$ and parameters $\epsilon\to0$.

The aim of the paper is to introduce and study an abstract condition concerning the choice $\epsilon = \epsilon(\tau)$ in order to obtain curves of maximal slope for the limit functional $\phi$. Moreover, we want to allow a relaxed form of minimization in each step of the scheme. \\

\paragraph{Curves of maximal slope and $\Gamma$-convergence}

The notion of curves of maximal slope goes back to \cite{DeGiorgiMarinoTosques80}, with further developments in \cite{Degiovanni-Marino-Tosques85}, \cite{MarinoSacconTosques89}.

Let $(\SS, d)$ be a complete metric space. Curves of maximal slope for an extended real functional $\phi: \SS \rightarrow (-\infty, +\infty]$ with respect to its relaxed slope are described by the energy dissipation inequality (EDI($\phi$))
\begin{equation*}
\phi(u(s)) - \phi(u(t)) \ \geq \ \frac{1}{2} \int^{t}_{s}{|\partial^- \phi|^2 (u(r)) \ dr} \ + \ \frac{1}{2} \int^{t}_{s}{|u'|^2 (r) \ dr}    
\end{equation*} 
for $\LL^1\text{-a.e. }s, t, \ s\leq t$, in which $|\partial^- \phi|$ denotes the relaxed slope and $|u'|$ the metric derivative (section \ref{subsec: gf}).     

In the finite dimensional and smooth setting, this corresponds to the gradient flow equation
\begin{equation*}
u'(t) = -\nabla \phi(u(t)).   
\end{equation*}
The term gradient flow is also common for curves of maximal slope.  \\

Let a sequence of energy functionals $(\phi_\epsilon)_{\epsilon > 0}$ $\Gamma$-converging to a limit functional $\phi$ be given. A natural question arises:\\

\textbf{If $u_\epsilon$ are curves of maximal slope for $\phi_\epsilon$ and they converge to a curve $u$, is $u$ then a curve of maximal slope for $\phi$?} ($\star$) \\

\textbf{1) The $\lambda$-convex case} Since the pioneering work on $G$-convergence for differential operators \cite{Spagnolo67}, \cite{de1973sulla}, the following statement is well-known:  If the functionals are $\lambda$-convex with equi-compact sublevel sets and $\SS$ is a Hilbert space, and the initial data converge, then the gradient flows $u_\epsilon$ converge to the limit gradient flow (in fact, Mosco convergence is sufficient in this case \cite{Attouch84}). A corresponding result can be proved for metric spaces as well \cite{MR3328994}. \\ 

\textbf{2) The Serfaty-Sandier approach} The considerations by Sandier and Serfaty in \cite{SandierSerfaty04}, \cite{serfaty2011gamma} are motivated by the convergence of the Ginzburg-Landau heat flow and the question of a general underlying structure in the cases in which a positive answer can be given to ($\star$).
Their main assumption is that the upper gradients $|\partial^-\phi_\epsilon|, \ |\partial^-\phi|$ satisfy
\begin{equation}\label{eq: Sandier Serfaty introduction}
u_\epsilon \to u \ \Rightarrow \ \mathop{\liminf}_{\epsilon\to 0} |\partial^- \phi_\epsilon|(u_\epsilon) \geq |\partial^- \phi|(u).  
\end{equation}    

A related problem for generalized gradient systems and rate-independent evolution is studied in \cite{mielke2014evolutionary}, \cite{MielkeRoubicekStefanelli08}, \cite{MRS12}.\\ 

In the $\lambda$-convex case, the following condition on the local slopes can be proved \cite{Ortner05TR}: 
\begin{equation}\label{eq: Ortner introduction}
u_\epsilon \to u \ \Rightarrow \ \mathop{\liminf}_{\epsilon \to 0} |\partial \phi_{\epsilon}|(u_\epsilon) \geq |\partial \phi|(u)  
\end{equation} 
(see section \ref{subsec: gf} for the definition of local slope). \\

So in both situations 1) and 2), the slopes satisfy a $\Gamma$-liminf condition. In general, such a condition does not hold. \\

\textbf{3) The general case}\\

In general, the answer to our question ($\star$) is \textbf{no}.
Even in the finite dimensional and smooth setting (as in Example \ref{ex: finite dimensional introduction}), the limit of a sequence of gradient flows $u_\epsilon$ for $\phi_\epsilon$ is in general no solution to the gradient flow equation for the limit functional $\phi$.\\

For illustrative purposes, we give a simple example. 
\begin{example}\label{ex: finite dimensional introduction}
We consider $f_\epsilon(x) = x^2 + \rho_\epsilon \cos^2\left(\frac{x}{\epsilon}\right) \ (\epsilon << \rho_\epsilon\to0)$ and $f(x)=x^2$ $(x\in\R)$. Let $u_\epsilon: [0, +\infty) \rightarrow \R$ satisfy
\begin{eqnarray*}
u_\epsilon'(t) &=& -f_\epsilon'(u_\epsilon(t)), \\
u_\epsilon(0) &=& u_\epsilon^0
\end{eqnarray*}
with initial values $u_\epsilon^0\to u_0\neq 0$. Then $(u_\epsilon)_{\epsilon > 0}$ converges pointwise to the constant curve $u \equiv u_0$ which does not solve the gradient flow equation for $f$. 
\end{example} 

The notion of $\Gamma$-convergence allows for a wide range of perturbations $\phi_{\epsilon}$ of $\phi$ so that a passage to the limit $\epsilon\to0$ in the energy dissipation inequality (EDI($\phi_\epsilon$)) $\longrightarrow$ (EDI($\phi$)) is in general not possible for lack of control over the upper gradients $|\partial^-\phi_\epsilon|$ of $\phi_\epsilon$.\\

Our approach aims to study the gradient flow motion along $(\phi_\epsilon)_{\epsilon > 0}$ on the level of the minimizing movement schemes and to establish a connection with the gradient flow motion of the limit functional $\phi$.   

\paragraph{Minimizing movements} 
At the beginning of the 90's, Ennio De Giorgi introduced the notion of minimizing movements \cite{DeGiorgi93} as ``natural meeting point'' of many different research fields in mathematics.

The minimizing movement scheme for gradient flows is given by
\begin{equation*}
\psi(u_\tau^n) + \frac{1}{2\tau} d^2(u_\tau^n, u_\tau^{n-1}) \ = \ \min_{v\in \SS} \left\{\psi(v) + \frac{1}{2\tau} d^2(v,u_\tau^{n-1}) \right\} 
\end{equation*}
with $n\in\N$ and time steps $\tau > 0$. It is closely related to the notion of curves of maximal slope. Under suitable assumptions, the piecewise constant interpolations of the discrete values $(u_\tau^n)_{n\in\N}$ converge (up to a subsequence) to a curve of maximal slope for $\psi$ (as $\tau\to0$) \cite{AmbrosioGigliSavare05}. The scheme mimics the gradient flow motion along the functional $\psi$ on a discrete level. (section \ref{subsec: MM})

For functionals $\phi_\epsilon \stackrel{\Gamma}{\to}\phi$, one can consider the minimizing movement scheme along $(\phi_\epsilon)_{\epsilon > 0}$, given by 
\begin{equation*}
\phi_{\epsilon}(u_{\tau, \epsilon}^n) + \frac{1}{2\tau} d^2(u_{\tau, \epsilon}^n, u_{\tau, \epsilon}^{n-1}) \ = \ \min_{v\in \SS} \left\{\phi_{\epsilon}(v) + \frac{1}{2\tau} d^2(v,u_{\tau, \epsilon}^{n-1}) \right\} 
\end{equation*}
for  $\tau, \epsilon > 0$ (with well-prepared initial values $u_{\tau, \epsilon}^0$), define the piecewise constant interpolations $u_{\tau, \epsilon}(t) \equiv u_{\tau, \epsilon}^n$ for $t\in((n-1)\tau, n\tau]$, and let $\tau$ and $\epsilon$ tend to $0$ simultaneously. \\  

We are interested in the cases in which the limit curves of the minimizing movement scheme along $(\phi_\epsilon)_{\epsilon > 0}$ are curves of maximal slope for $\phi$. \\ 

In (\cite{braides2012local}, chapter 8) such minimization problems are examined for concrete examples of functionals $\phi_\epsilon \stackrel{\Gamma}{\to} \phi$, highlighting the dependence of the asymptotic behaviour on the interaction between $\epsilon$ and $\tau$.

Moreover, under suitable assumptions on $\phi_\epsilon$, the following statement can be proved (\cite{braides2012local}, Theorem 8.1) by means of the Fundamental Theorem of $\Gamma$-convergence and a diagonal argument: \\

If discrete values $(u_{\tau, \epsilon}^n)_{n\in\N_0} \ (\tau, \epsilon > 0)$ are given in accordance to the scheme above, then 
\begin{itemize}
\item there exists $\epsilon = \epsilon(\tau)$ such that any limit curve of the corresponding piecewise constant interpolations $(u_{\tau, \epsilon(\tau)})_{\tau > 0}$ can also be obtained by the minimizing movement scheme along the single functional $\phi$,
\item there exists $\tau = \tau(\epsilon)$ such that any limit curve of $(u_{\tau(\epsilon), \epsilon})_{\epsilon > 0}$ is the limit of a sequence of curves of maximal slope for $\phi_\epsilon$.  
\end{itemize}   

In view of this result and the examples in \cite{braides2012local}, we notice that the interaction between the time step $\tau$ and the parameter $\epsilon$ should be crucial in order to achieve convergence of the scheme to a curve of maximal slope for $\phi$. This goes hand in hand with our considerations on why the passage to the limit in the energy dissipation inequality is in general not possible.  

A special case in which the desired convergence of the scheme holds independent of the interaction between $\epsilon$ and $\tau$ is considered by Ortner in \cite{Ortner05TR}, by extending the arguments of \cite{AmbrosioGigliSavare05}. His main assumption is the $\Gamma$-liminf condition (\ref{eq: Ortner introduction}) on the local slopes which can be viewed as discrete counterpart of the Serfaty-Sandier condition (\ref{eq: Sandier Serfaty introduction}).    

\paragraph{Aim of our paper} This paper aims to derive a general condition concerning the choice $\epsilon = \epsilon(\tau)$ in order to achieve convergence of the minimizing movement scheme along $(\phi_\epsilon)_{\epsilon > 0}$ to curves of maximal slope for $\phi$. Note that we interpret the gradient flow motion along $(\phi_\epsilon)_{\epsilon > 0}$ as joint steepest descent movement instead of considering the single energy dissipation inequalities (EDI($\phi_\epsilon$)): 
We pass to the limits $\tau\to0$ and $\epsilon\to0$ simultaneously in the minimizing movement scheme, with a suitable choice $\epsilon = \epsilon(\tau)$. By comparison, the study of the limit behaviour of (EDI($\phi_\epsilon$)) (as in \cite{SandierSerfaty04}, \cite{serfaty2011gamma}) is related to first passing to the limit $\tau\to0$  in the scheme for fixed $\epsilon > 0$ and only then passing to the limit $\epsilon\to0$.

A special feature of our theory is that we are able to \textit{relax} the minimizing movement scheme along $(\phi_\epsilon)_{\epsilon > 0}$ by allowing \textit{approximate minimizers} in each step. We do not need the existence of exact solutions to the minimization problems. In particular, we do not need to require any lower semicontinuity or compactness property of the single functionals $\phi_\epsilon$.

As a particular case, the paper also deals with a relaxation of the classical minimizing movement scheme along a single functional.\\ 

Let us introduce a relaxed minimizing movement scheme along $(\phi_\epsilon)_{\epsilon > 0}$. We associate every time step $\tau > 0$ with a parameter $\epsilon = \epsilon(\tau) \to0 \ (\tau\to0)$ in such a way that $\epsilon(\tau)$ converges to $0$ as $\tau\to 0$.

\paragraph{Relaxed minimizing movement scheme along $(\phi_\epsilon)_{\epsilon > 0}$}
For every time step $\tau > 0$, find a sequence $\left(u_\tau^n\right)_{n\in\N}$ by the following scheme.
 
The sequence of initial values $u_\tau^0 \to u^0 \in D(\phi)$ satisfies $\phi_{\epsilon(\tau)}(u_\tau^0) \to \phi(u^0)$
and $u_\tau^n\in\SS \ (n\in\N)$ satisfies
\begin{equation}\label{eq: relaxed form of minimi introduction}
\phi_{\epsilon(\tau)}(u_\tau^n) + \frac{1}{2\tau} d^2(u_\tau^n, u_\tau^{n-1}) \ \leq \ \inf_{v\in \SS} \left\{\phi_{\epsilon(\tau)}(v) + \frac{1}{2\tau} d^2(v,u_\tau^{n-1}) \right\} + \gamma_\tau \tau
\end{equation}
with some $\gamma_\tau > 0, \ \gamma_\tau\to 0$ as $\tau\to0$.\\

The reader might be interested in the question of whether we may allow an error $\gamma_\tau^{(n)}$ depending on $n\in\N$ (instead of $\gamma_\tau\tau$) in the approximate minimization problems (\ref{eq: relaxed form of minimi introduction}). Indeed, a generalization of our theory to a non-uniform distribution of the error is possible.\\

\paragraph{A general condition concerning the right choice $\epsilon = \epsilon(\tau)$ in (\ref{eq: relaxed form of minimi introduction})}

Our main assumption (see assumption \ref{ass: crucial assumption}) is as follows: we suppose that for all $u_\tau, u\in\SS$ with $u_\tau \to u, \ \sup_\tau \phi_{\epsilon(\tau)}(u_\tau) < +\infty$ it holds that 
\begin{equation}\label{eq: crucial assumption introduction}
\mathop{\liminf}_{\tau \to 0} \frac{\phi_{\epsilon(\tau)}(u_\tau) - \cY_\tau \phi_{\epsilon(\tau)}(u_\tau)}{\tau} \ \geq \ \frac{1}{2} |\partial^- \phi|^2 (u),
\end{equation}
in which 
\begin{equation*}
\cY_\tau \phi_{\epsilon(\tau)} (u_\tau) := \inf_{v\in \SS} \left\{\phi_{\epsilon(\tau)}(v) + \frac{1}{2\tau} d^2(v,u_\tau) \right\}
\end{equation*}
denotes the Moreau-Yosida approximation. 
Our condition (\ref{eq: crucial assumption introduction}) relates the diagonal steepest descent movement along the sequence $(\phi_\epsilon)_{\epsilon > 0}$ with the relaxed slope of $\phi$.\\ 

Under suitable natural coercivity assumptions, we prove  
\begin{itemize}
\item
if condition (\ref{eq: crucial assumption introduction}) holds for $\epsilon = \epsilon(\tau)$, then any limit curve of our relaxed minimizing movement scheme along $(\phi_\epsilon)_{\epsilon > 0}$ with choice $\epsilon = \epsilon(\tau)$ is a curve of maximal slope for $\phi$. (see Theorem \ref{thm: main theorem}, Proposition \ref{prop: supplement to the theorem})
\end{itemize} 

\paragraph{Existence of a right choice $\epsilon = \epsilon(\tau)$} 
If $(\SS, d)$ is separable, we can prove that 
\begin{itemize}
\item  there always exists a sequence $(\epsilon_\tau)_{\tau > 0}$ with $\epsilon_\tau > 0$ such that our main assumption (\ref{eq: crucial assumption introduction}) is satisfied for all choices $(\epsilon(\tau))_{\tau > 0}$ with $\epsilon(\tau) \leq \epsilon_\tau$ $(\epsilon(\tau)\to0)$. (see Theorem \ref{thm: existence of epsilon(tau)}) 
\end{itemize}
We notice that this choice $\epsilon = \epsilon(\tau)$ solely depends on the velocity of $\Gamma$-convergence $\phi_\epsilon \stackrel{\Gamma}{\to} \phi$. 

\paragraph{A general example for the choice $\epsilon = \epsilon(\tau)$} We prove under suitable natural coercivity assumptions that  
if the local boundedness condition  
\begin{equation}\label{eq: local boundedness condition}
\mathop{\limsup}_{n\to+\infty} \sup_{v: d(v,u)\leq\rho_n} \left|\frac{\phi(v) - \phi_{\epsilon_n}(v)}{\epsilon_n}\right| \ < \ +\infty
\end{equation} 
holds for every $u\in\SS$, $\rho_n\to0 \ (\rho_n > 0)$, then (\ref{eq: crucial assumption introduction}) is satisfied for every choice $\epsilon = \epsilon(\tau)$ with $\frac{\epsilon(\tau)}{\tau}\to 0$ as $\tau\to0$. (see Proposition \ref{prop: perturbations})   
 
\paragraph{Plan of the paper and further results} 
In section \ref{sec:prel}, we give basic definitions. In \ref{subsec:top ass}, we specify our topological assumptions (which allow for metric topological spaces $(\SS, d, \sigma)$ where the topology $\sigma$ does not necessarily coincide with the one induced by the metric $d$).\\
   
In section \ref{sec: main theorem}, we introduce the relaxed minimizing movement scheme (\ref{eq: relaxed form of minimi introduction}) along $(\phi_\epsilon)_{\epsilon > 0}$ and we prove the convergence of the scheme to the gradient flow motion of the limit functional $\phi$ under the main assumption (\ref{eq: crucial assumption introduction}). We discuss a generalization of our theory to a non-uniform distribution of the error in (\ref{eq: relaxed form of minimi introduction}).\\ 

In section \ref{sec: constant sequences}, we focus on the special case in which $\phi_\epsilon = \psi$ is independent of $\epsilon > 0$. Applying the results of section \ref{sec: main theorem}, we show that one may allow a relaxed form of minimization in the clasical scheme along a single functional.
Moreover, we notice that the error order $o(\tau)$, characterizing the asymptotic behaviour of $\gamma_\tau\tau$ in the definition of the relaxed minimizing movement scheme (\ref{eq: relaxed form of minimi introduction}), is optimal. We also obtain a lower semicontinuous envelope relaxation result.\\ 

In section \ref{sec: convergence of gradient flows}, we consider the case that $(\phi_\epsilon)_{\epsilon > 0}, \phi$ satisfy the Serfaty-Sandier condition (\ref{eq: Sandier Serfaty introduction}) or its discrete counterpart (\ref{eq: Ortner introduction}). We prove that in this situation, condition (\ref{eq: crucial assumption introduction}) is satisfied for every choice $\epsilon = \epsilon(\tau)$.\\

In section \ref{sec: in separable metric spaces}, we prove, for the case that $(\SS, d)$ is separable, the existence of a right choice $\epsilon = \epsilon(\tau)$ with regard to condition (\ref{eq: crucial assumption introduction}).\\

In sections \ref{sec: 1D examples} and \ref{sec: perturbations and co}, we illustrate general methods to determine $\epsilon = \epsilon(\tau)$ with regard to (\ref{eq: crucial assumption introduction}). 

In \ref{subsec: perturbations}, we determine $\epsilon = \epsilon(\tau)$ for the case that the general condition (\ref{eq: local boundedness condition}) holds.
 
In \ref{subsec: restriction to bounded subsets}, we show explicitly that the deciding factor in the whole theory is the local behaviour of the $\Gamma$-converging sequence of functionals. 

In \ref{subsec: time space discretizations}, we consider time-space discretizations for a single functional, i.e. we set $\phi_\epsilon = \phi + \I_{\WW_\epsilon}$ and we derive conditions on $\WW_{\epsilon(\tau)}\subset\SS$ such that condition (\ref{eq: crucial assumption introduction}) holds.\\ 

In section \ref{sec: some final remarks}, we mention possible generalizations of our theory.   

\newpage

\section{Preliminaries}
\label{sec:prel}

\subsection{Gradient flows in metric spaces}
\label{subsec: gf}
For an introduction to gradient flows in metric spaces we refer to the fundamental book by Ambrosio, Gigli and Savar\'e \cite{AmbrosioGigliSavare05}. In this section we give a brief overview of some of the basic definitions which can be found in detail in (\cite{AmbrosioGigliSavare05}, chapter 1 and 2).  \\

Let $(a,b)$ be an open (possibly unbounded) interval of $\R$.\\

In the finite dimensional case, the classical gradient flow equation 
\begin{equation*}
u'(t) = -\nabla f(u(t)), \ t\in (a,b)
\end{equation*} 
for a $C^1$-function $f: \Rd \rightarrow \R$ can be equivalently expressed by the equation
\begin{equation*}
(f\circ u)'(t) = -\frac{1}{2}|u'|^2(t)-\frac{1}{2}\left|\nabla f(u(t))\right|^2, \ t\in (a,b).
\end{equation*}
This equivalent formulation constitutes a heuristic starting point for the introduction of gradient flows in general metric spaces $(\SS, d)$, the so-called curves of maximal slope, for extended real functionals $\psi: \SS \rightarrow (-\infty, +\infty]$.

\paragraph{Absolutely continuous curves, relaxed slope}
\begin{definition}
Let $(\SS, d)$ be a complete metric space. We say that a curve $v: \ (a,b) \rightarrow \SS$ is absolutely continuous and write $v\in AC(a,b;\SS)$ if there exists $m\in L^1(a,b)$ such that 
\begin{equation*}
d(v(s),v(t)) \leq \int^{t}_{s}{m(r) \ dr} \ \ \forall a < s\leq t < b. 
\end{equation*}\\
In this case, the limit
\begin{equation*}
|v'|(t) := \mathop{\lim}_{s\to t} \frac{d(v(s),v(t))}{|s-t|}
\end{equation*}
exists for $\LL^1$-a.e. $t\in (a,b)$, the function $t \mapsto |v'|(t)$ belongs to $L^1(a,b)$ and is called the metric derivative of $v$. The metric derivative is $\LL^1$-a.e. the smallest admissible function $m$ in the definition above. 
\end{definition}

Let $\psi: \SS \rightarrow (-\infty, +\infty]$ be an extended real functional with proper effective domain 
\begin{equation*}
D(\psi) := \{\psi < +\infty\} \ \neq \emptyset.
\end{equation*}
For $v\in D(\psi)$ we define
\begin{equation*}
|\partial\psi|(v) := \mathop{\limsup}_{w\stackrel{d}{\to}v} \frac{(\psi(v)-\psi(w))^+}{d(v,w)}.
\end{equation*}
Following \cite{AmbrosioGigliSavare05}, we consider a slight modification of the sequentially $\sigma$-lower semicontinuous envelope of $|\partial\psi|$ with respect to a suitable topology $\sigma$ on $\SS$ (see section \ref{subsec:top ass} for the topological assumptions). 

\begin{definition} 
The relaxed slope $|\partial^- \psi|: \SS \rightarrow [0, +\infty]$ of $\psi$ is defined by 
\begin{equation*}
|\partial^- \psi|(u) := \inf  \left\{\mathop{\liminf}_{n\to \infty} |\partial\psi|(u_n): \ u_n \stackrel{\sigma}{\weakto} u, \ \sup_{n}\{d(u_n, u), \psi(u_n)\} < +\infty \right\}.
\end{equation*}
\end{definition}

The concept of strong and weak upper gradients can be viewed as a weak counterpart of the modulus of the gradient in a general metric and non-smooth setting. 
 
\begin{definition}
The relaxed slope $|\partial^- \psi|$ is a strong upper gradient for the functional $\psi$ if for every absolutely continuous curve $v\in AC(a,b;\SS)$ the function $|\partial^- \psi|\circ v$ is Borel and 
\begin{equation*}
|\psi(v(t)) - \psi(v(s))| \leq \int^{t}_{s}{|\partial^- \psi|(v(r))|v'|(r) \ dr} \ \ \forall a < s\leq t < b.
\end{equation*}
In particular, if $|\partial^- \psi|\circ v |v'| \in L^1(a,b)$ then $\psi \circ v$ is absolutely continuous and 
\begin{equation*}
|(\psi \circ v)'(t)| \leq |\partial^- \psi|(v(t))|v'|(t) \text{ for } \LL^1\text{-a.e. } t\in (a,b).
\end{equation*}
\end{definition}

\begin{definition}
The relaxed slope $|\partial^- \psi|$ is a weak upper gradient for the functional $\psi$ if for every absolutely continuous curve $v\in AC(a,b;\SS)$ such that
$\psi \circ v$ is $\LL^1$-a.e. equal in $(a,b)$ to a function $\varphi$ with finite pointwise variation in $(a,b)$ and 
$|\partial^- \psi|\circ v |v'| \in L^1(a,b)$,
it holds that 
\begin{equation*}
|\varphi'(t)|\leq |\partial^- \psi|(v(t))|v'|(t) \text{ for } \LL^1{-a.e.} t\in (a,b).
\end{equation*}
\end{definition}

Now, we have all the ingredients to define curves of maximal slope. 

\paragraph{Curves of maximal slope}  
\begin{definition}
Let $|\partial^- \psi|$ be a strong or weak upper gradient for the functional $\psi: \SS\rightarrow (-\infty, +\infty]$. \\ A locally absolutely continuous curve $u: (a,b) \to \SS$ is called a curve of maximal slope for $\psi$ with respect to its upper gradient $|\partial^- \psi|$ if the following energy dissipation inequality is satisfied for $\LL^1\text{-a.e. } s, t\in (a,b), \ s\leq t$: 
\begin{equation}\label{eq: EDI}
\psi(u(s)) - \psi(u(t)) \ \geq \ \frac{1}{2} \int^{t}_{s}{|\partial^- \psi|^2 (u(r)) \ dr} \ + \ \frac{1}{2} \int^{t}_{s}{|u'|^2 (r) \ dr}.  
\end{equation}  
\end{definition}

Under additional assumptions, the existence of curves of maximal slope with respect to the relaxed slope can be proved \cite{AmbrosioGigliSavare05}. Their proof is based on the minimizing movement scheme for gradient flows which we explain in section \ref{subsec: MM}.\\

Before outlining the connection between curves of maximal slope and the notion of minimizing movements, we specify the assumptions on the metric space $(\SS, d)$ and the topology $\sigma$ on it.    

\subsection{Topological assumptions}
\label{subsec:top ass}

Throughout this paper 
\begin{eqnarray}
(\SS, d)  \text{ will be a given complete metric space}
\end{eqnarray}
and
\begin{eqnarray*}
\sigma \text{ will be an Hausdorff topology on } \SS \text{ compatible with } d \text{:}
\end{eqnarray*}
\begin{eqnarray}
(u_n,v_n) \stackrel{\sigma}{\weakto} (u,v)  \ &\Rightarrow& \ \mathop{\liminf}_{n \to \infty} d(u_n,v_n) \geq d(u,v) \text{ ,} \\
d(u_n,v_n) \rightarrow 0 \text{, } u_n \stackrel{\sigma}{\weakto} u \ &\Rightarrow& \  v_n \stackrel{\sigma}{\weakto} v \text{.}
\end{eqnarray}

These are exactly the same topological assumptions as in \cite{AmbrosioGigliSavare05}. The reasons for which they are made are explained in \cite{AmbrosioGigliSavare05} and can also be observed in this paper. 

\subsection{Minimizing movement scheme for gradient flows}
\label{subsec: MM}
The notion of minimizing movements was established by Ennio de Giorgi \cite{DeGiorgi93} at the beginning of the 90's, who got his inspiration from the paper \cite{AlmgrenTaylorWang93}. 

\begin{definition}
Let $F: (0,1)\times \SS \times \SS \rightarrow [-\infty, +\infty]$ be given and initial values $u_{\tau}^{0} \stackrel{\sigma}{\weakto} u^0 \ (\tau\to 0)$. If we can define $\overline{u}_{\tau}: [0, +\infty) \rightarrow \SS$, $\overline{u}_{\tau}(0) = u_{\tau}^{0}$, in the following way
\begin{equation*}
\overline{u}_{\tau}(t) \equiv u_{\tau}^{n} \text{ for } t\in ((n-1)\tau, n\tau] 
\end{equation*}
and 
\begin{equation*}
u_{\tau}^{n} \text{ is a minimizer for } F(\tau, u_{\tau}^{n-1}, \cdot), \ n\in \N, 
\end{equation*}
and if (a subsequence of) $(\overline{u}_{\tau})_{\tau > 0}$ $\sigma$-converges pointwise in $[0, +\infty)$ to a curve $u: [0, +\infty) \rightarrow \SS$, then $u$ is called a (generalized) minimizing movement for $F$ to the initial value $u^0$.  
\end{definition}

If we apply the minimizing movement scheme to 
\begin{equation*}
F(\tau,u,v) := f(v) + \frac{1}{2\tau}|u-v|^2 \ (u,v \in \Rd)
\end{equation*}
with $f: \Rd \rightarrow \R$ a Lipschitz continuous $C^1$-function, then the necessary condition of first order leads to a discrete version of the gradient flow equation for $f$, and indeed, every (generalized) minimizing movement for F is a gradient flow for $f$. 

With this in mind, for a given functional $\psi: \SS \rightarrow (-\infty, +\infty]$ we define 
\begin{equation*}
F_{\psi}: (0,1)\times \SS \times \SS \rightarrow [-\infty, +\infty], \ F_{\psi}(\tau, u, v) := \psi(v) + \frac{1}{2\tau} d^2(v,u) 
\end{equation*}
which seems to be only natural in order to construct steepest descent curves for $\psi$.\\ 

A related object of study is the Moreau-Yosida approximation defined below. We will repeatedly use the Moreau-Yosida approximation and the following notation.

\paragraph{Moreau-Yosida approximation}

Let $\tau > 0$ be given. The Moreau-Yosida approximation $\cY_\tau \psi$ of a functional $\psi: \SS \rightarrow (-\infty, +\infty]$ is defined as 
\begin{equation}
\cY_\tau \psi (u) := \inf_{v\in \SS} \left\{\psi(v) + \frac{1}{2\tau} d^2(v,u) \right\}, \ u\in \SS. 
\end{equation}

We define $J_\tau \psi[u]$ as the corresponding set of minimizers, i.e.  
\begin{eqnarray}
u_\tau \in J_\tau \psi[u] \ :\Leftrightarrow \ \psi(u_\tau) + \frac{1}{2\tau} d^2(u_\tau, u) = \cY_\tau \psi(u).
\end{eqnarray} 

\paragraph{Existence of curves of maximal slope}
Let $\psi: \SS \rightarrow (-\infty, +\infty]$ be given.\\

The following existence result is proved in (\cite{AmbrosioGigliSavare05}, chapter 2 and 3): \\

If $\psi$ satisfies assumption \ref{ass: AGS1}, then the set of generalized minimizing movements for $F_\psi$ is non-empty and for every generalized minimizing movement $u: [0, +\infty) \rightarrow \SS$ the energy inequality 
\begin{equation}\label{eq: AGS energy inequality}
\psi(u(0)) - \psi(u(t)) \ \geq \ \frac{1}{2} \int^{t}_{0}{|\partial^- \psi|^2 (u(r)) \ dr} \ + \ \frac{1}{2} \int^{t}_{0}{|u'|^2 (r) \ dr} 
\end{equation}
holds for all $t > 0$.\\ 

\begin{assumption}\label{ass: AGS1} 
\upshape
We suppose that there exist $A, B > 0, \ u_{\star}\in \SS$ such that 
\begin{equation}\label{eq: AGS1 coercive}  
\psi(\cdot) \geq -A -B d^2(\cdot, u_{\star}).
\end{equation}
Moreover, we assume that the following holds:
\begin{equation}\label{eq: AGS1 lower semicontinuity}
\sup_{n,m} d(u_n, u_m) < +\infty, \ u_n \stackrel{\sigma}{\weakto} u \ \Rightarrow \ \mathop{\liminf}_{n\to \infty} \psi(u_n) \geq \psi(u)
\end{equation}
and  
\begin{equation}\label{eq: AGS1 compact}
\sup_{n,m}\{d(u_n,u_m), \psi(u_n)\} < +\infty \ \Rightarrow \ \exists \ n_k \uparrow +\infty, u\in\SS: \ u_{n_k}\stackrel{\sigma}{\weakto}u. 
\end{equation}
\end{assumption}

In particular, if $|\partial^- \psi|$ is a strong upper gradient, equality holds in (\ref{eq: AGS energy inequality}) and $u$ is a curve of maximal slope for $\psi$ with respect to $|\partial^- \psi|$.\\

If $|\partial^- \psi|$ is only a weak upper gradient for $\psi$ and we assume in addition that $\psi$ satisfies the following continuity condition
\begin{equation}\label{eq: continuity condition}
\sup_{n\in\N}\left\{|\partial \psi|(v_n), d(v_n,v), \psi(v_n)\right\} < +\infty, \ v_n\stackrel{\sigma}{\weakto}v \ \Rightarrow \ \psi(v_n)\to \psi(v),
\end{equation}
then the proof of the energy inequality (\ref{eq: AGS energy inequality}) can be extended and the energy dissipation inequality (\ref{eq: EDI}) be obtained.  \\  

Strategies to prove the conditions on the relaxed slope to be an upper gradient and the continuity condition (\ref{eq: continuity condition}) respectively and examples in which they are satisfied are expounded e.g. in \cite{AmbrosioGigliSavare05}, \cite{RossiSavare06}. 

\subsection{$\Gamma$-convergence}

$\Gamma$-convergence was introduced by Ennio de Giorgi in the early 70's.  

\begin{definition}
Let $X$ be a topological space. A sequence $f_j : X\rightarrow [-\infty, +\infty]$ sequentially $\Gamma$-converges in $X$ to $f_\infty : X\rightarrow [-\infty, +\infty]$ if for all $x\in X$, $x_j\to x$ the following liminf-inequality holds
\begin{equation*}
f_\infty (x) \leq \mathop{\liminf}_{j\to \infty} f_j(x_j) 
\end{equation*}
and if for all $x\in X$ there exists a recovery sequence $\tilde{x}_j \to x$ such that
\begin{equation*}
f_\infty(x) = \mathop{\lim}_{j\to \infty} f_j(\tilde{x}_j).
\end{equation*}
\end{definition}

For its various applications, see for example the introductory book by Braides \cite{Braides02}. 

\newpage 

\section{Main theorem}\label{sec: main theorem}

We systematically study the gradient flow motion along a sequence of functionals $\phi_\epsilon : \SS \rightarrow (-\infty, +\infty]$ on the level of the minimizing movement scheme in each step of which we are able to allow a relaxed form of minimization. The functionals $\phi_\epsilon$ are associated with a limit functional $\phi: \SS \rightarrow (-\infty, +\infty]$ through a weakened $\Gamma(\sigma)$-liminf-inequality and a recovery sequence of initial values.

\subsection{Minimizing movement: $\Gamma$-convergence, relaxations} \label{subsec: thm}

Let $\phi, \ \phi_\epsilon: \SS \rightarrow (-\infty, +\infty], \ D(\phi), \ D(\phi_\epsilon) \neq \emptyset \ (\epsilon > 0)$ be given. 

We are dealing with the following assumptions on $(\phi_\epsilon)_{\epsilon > 0}$ and the relation between $(\phi_\epsilon)_{\epsilon > 0}$ and $\phi$:

\begin{assumption}\label{ass: co}
\upshape
We suppose that there exist $A, B > 0, \ u_{\star}\in \SS$ such that 
\begin{equation} \label{eq: coercive} 
\phi_\epsilon(\cdot) \geq -A -B d^2(\cdot, u_{\star}) \text{ for all } \epsilon > 0.
\end{equation}
Moreover, we assume that for $\epsilon_n\to 0 \ (\epsilon_n > 0)$ it holds that
\begin{equation} \label{eq: compact}
\sup_{n,m} \left\{\phi_{\epsilon_n}(u_{\epsilon_n}), d(u_{\epsilon_n},u_{\epsilon_m})\right\} < +\infty \ \Rightarrow \ \exists \ n_k \uparrow +\infty, u\in \SS: \ u_{\epsilon_{n_k}}\stackrel{\sigma}{\weakto}u. 
\end{equation}
\end{assumption}

\begin{assumption}\label{ass: liminf}
\upshape
We suppose that $(\phi_\epsilon)_{\epsilon > 0}$ and $\phi$ are connected through a weakened $\Gamma(\sigma)$-liminf-inequality, i.e. for $\epsilon_n\to 0 \ (\epsilon_n > 0)$ it holds that 
\begin{equation}
\sup_{n,m} d(u_{\epsilon_n}, u_{\epsilon_m}) < +\infty, \ u_{\epsilon_n} \stackrel{\sigma}{\weakto} u \ \Rightarrow \ \mathop{\liminf}_{n\to\infty} \phi_{\epsilon_n}(u_{\epsilon_n}) \geq \phi(u).
\end{equation} 
\end{assumption}

In view of the theory developed in \cite{AmbrosioGigliSavare05} (see section \ref{subsec: MM} in this paper), our assumptions \ref{ass: co} and \ref{ass: liminf} arise quite naturally.\\
 
Let us define a minimizing movement scheme for the gradient flow motion along the sequence $(\phi_\epsilon)_{\epsilon > 0}$ with a relaxed form of minimization in each step and time steps $\tau \to 0$.
 
We associate every time step $\tau > 0$ in the scheme with $\epsilon = \epsilon(\tau) > 0$ in such a way that $\epsilon(\tau)$ converges to $0$ as $\tau\to 0$. This choice $\epsilon= \epsilon(\tau)$ is crucial, as we will see. 

\paragraph{Relaxed minimizing movement scheme along $(\phi_\epsilon)_{\epsilon > 0}$}

For every time step $\tau > 0$, find a sequence $(u_\tau^n)_{n\in\N}$ by the following scheme.
 
The sequence of initial values $u_\tau^0 \stackrel{\sigma}{\weakto} u^0 \in D(\phi) \ (\tau\to 0)$ satisfies
\begin{equation}\label{eq: recovery sequence}
\sup_\tau d(u_\tau^0, u^0) < +\infty, \ \phi_{\epsilon(\tau)}(u_\tau^0) \to \phi(u^0)
\end{equation}
and $u_\tau^n\in\SS \ (n\in\N)$ satisfies
\begin{equation}\label{eq: relaxed form of minimi}
\phi_{\epsilon(\tau)}(u_\tau^n) + \frac{1}{2\tau} d^2(u_\tau^n, u_\tau^{n-1}) \ \leq \ \inf_{v\in \SS} \left\{\phi_{\epsilon(\tau)}(v) + \frac{1}{2\tau} d^2(v,u_\tau^{n-1}) \right\} + \gamma_\tau \tau
\end{equation}
with some $\gamma_\tau > 0, \ \gamma_\tau \to 0$ as $\tau \to 0$. \\

Let $\overline{u}_\tau : [0, +\infty) \rightarrow \SS$ be the corresponding piecewise constant interpolation, i.e.
\begin{eqnarray}
\overline{u}_\tau (t) &\equiv& u_\tau^n \text{    if } t\in ((n-1)\tau, n\tau], \ n\in \N, \label{eqn: mm} \\
\overline{u}_\tau (0) &=& u_\tau^0 \label{eqn: initial value}.
\end{eqnarray}

\paragraph{Main assumption}
The next assumption on the interrelation between the relaxed slope $|\partial^- \phi|$ of $\phi$ and the relaxed minimizing movement scheme along $(\phi_\epsilon)_{\epsilon > 0}$ is playing a central role throughout this paper.  

\begin{assumption}\label{ass: crucial assumption}
\upshape
We suppose that for all $u, u_\tau\in \SS \ (\tau > 0)$ such that 
\begin{equation*}
u_\tau \stackrel{\sigma}{\weakto} u, \ \sup_\tau \left\{\phi_{\epsilon(\tau)}(u_\tau), d(u_\tau, u)\right\} < +\infty 
\end{equation*}
it holds that 
\begin{equation}\label{eq: crucial assumption}
\mathop{\liminf}_{\tau \to 0} \frac{\phi_{\epsilon(\tau)}(u_\tau) - \cY_\tau \phi_{\epsilon(\tau)}(u_\tau)}{\tau} \ \geq \ \frac{1}{2} |\partial^- \phi|^2 (u).
\end{equation}
\end{assumption}

As defined in section \ref{subsec: MM}, $\cY_\tau\phi_{\epsilon(\tau)}$ denotes the Moreau-Yosida approximation of the functional $\phi_{\epsilon(\tau)}$.

\newpage
Now, our theorem reads as follows.

\begin{theorem}\label{thm: main theorem}
Let the assumptions \ref{ass: co}, \ref{ass: liminf} and \ref{ass: crucial assumption} be satisfied and construct $\overline{u}_\tau : [0, +\infty) \rightarrow \SS \ (\tau > 0)$ according to (\ref{eq: recovery sequence}) - (\ref{eqn: initial value}). 

Then there exist a locally absolutely continuous curve $u: [0, +\infty) \rightarrow \SS$ and a subsequence $(\overline{u}_{\tau_k})_{k\in \N}$ such that 
\begin{equation}\label{eq: convergence}
\overline{u}_{\tau_k} (t) \stackrel{\sigma}{\weakto} u(t) \text{ for all } t \geq 0
\end{equation} 
and $u$ satisfies the initial condition
\begin{equation}\label{eq: initial condition}
u(0) = u^0
\end{equation}
and the energy inequality
\begin{equation}\label{eq: energy inequality}
\phi(u(0)) - \phi(u(t)) \ \geq \ \frac{1}{2} \int^{t}_{0}{|\partial^- \phi|^2 (u(s)) \ ds} \ + \ \frac{1}{2} \int^{t}_{0}{|u'|^2 (s) \ ds}
\end{equation}
for all $t\geq 0$.

In particular, if the relaxed slope of $\phi$ is a strong upper gradient for $\phi$, then equality holds in (\ref{eq: energy inequality}) and  
\begin{equation}
u \text{ is a curve of maximal slope for } \phi \text{ with respect to } |\partial^- \phi|. 
\end{equation}
\end{theorem}

\paragraph{Some comments on our assumptions}

Please be aware of the fact that we do not require any lower semicontinuity or compactness property of the single functionals $\phi_\epsilon$. Since we are allowing a relaxed form of minimization (\ref{eq: relaxed form of minimi}) in our scheme, the coercivity (\ref{eq: coercive}) of every functional $\phi_\epsilon$ is sufficient to guarantee the existence of the curves $\overline{u}_\tau$ for small $\tau > 0$. 

Some equi-coercivity and combined compactness assumption on $(\phi_\epsilon)_{\epsilon > 0}$ such as (\ref{eq: coercive}) and (\ref{eq: compact}) in assumption \ref{ass: co} are needed in order to prove the existence of a subsequence $(\overline{u}_{\tau_k})_{k\in \N}$ pointwise $\sigma$-converging to a locally absolutely continuous curve. 

The energy inequality (\ref{eq: energy inequality}) can then be proved by using assumptions \ref{ass: liminf} and \ref{ass: crucial assumption} which make the passage to $\phi$ and $|\partial^- \phi|$ possible. Some kind of assumption appropriately connecting the relaxed slope of $\phi$ with the energy driven motion along the sequence $(\phi_{\epsilon(\tau)})_{\tau > 0}$ with time steps $\tau > 0$ is necessary and our assumption \ref{ass: crucial assumption} will turn out to be a good choice.

\subsection{Proof}

We prove Theorem \ref{thm: main theorem}.

\begin{proof}
The proof divides into three main steps. \\

\framebox{\textbf{Existence of $u$}} The proof of the existence of a subsequence $(\overline{u}_{\tau_k})_{k\in \N}$ pointwise $\sigma$-converging to a locally absolutely continuous curve $u$ follows similar arguments as the proof of the existence of a generalized minimizing movement in the classical scheme along a single functional $\psi$ (\cite{AmbrosioGigliSavare05}, chapter 3). 

We have already noticed that for $n\in\N$ and small $\tau > 0$ there exist $u_\tau^n \in \SS$ satisfying (\ref{eq: relaxed form of minimi}). Moreover, for every $\delta > 0$ it holds that
\begin{eqnarray*}
&&\frac{1}{2} d^2(u_\tau^n, u_\star) - \frac{1}{2} d^2(u_\tau^0, u_\star) = \sum_{j=1}^{n}{\left(\frac{1}{2}d^2(u_\tau^j,u_\star) -\frac{1}{2}d^2(u_\tau^{j-1},u_\star)\right)} \\
&\leq& \sum_{j=1}^{n}{d(u_\tau^j, u_\tau^{j-1}) d(u_\tau^j,u_\star)} \\
&\leq& \delta \sum_{j=1}^{n}{\frac{d^2(u_\tau^j, u_\tau^{j-1})}{2\tau}} \ + \ \frac{1}{2\delta} \sum_{j=1}^{n}{\tau d^2(u_\tau^j, u_\star)} \\
&\leq& \delta \sum_{j=1}^{n}{\left(\phi_{\epsilon(\tau)}(u_\tau^{j-1}) - \phi_{\epsilon(\tau)}(u_\tau^j) + \gamma_\tau \tau \right)} \ + \ \frac{1}{2\delta}\sum_{j=1}^{n}{\tau d^2(u_\tau^j, u_\star)} \\
&\leq& \delta (\phi_{\epsilon(\tau)}(u_\tau^0) + A + B d^2(u_\tau^n, u_\star) + \gamma_\tau n \tau) \ + \ \frac{1}{2\delta} \sum_{j=1}^{n}{\tau d^2(u_\tau^j, u_\star)}.
\end{eqnarray*}
Choose $\delta := \frac{1}{4B}$. Then we have 
\begin{equation*}
d^2(u_\tau^n, u_\star) \ \leq \ 2d^2(u_\tau^0, u_\star) + \frac{1}{B}\phi_{\epsilon(\tau)}(u_\tau^0) + \frac{A}{B} + \frac{\gamma_\tau n \tau}{B} \ + \ \frac{B}{2} \sum_{j=1}^{n}{\tau d^2(u_\tau^j, u_\star)}.
\end{equation*}
By applying the discrete version of Gronwall lemma stated below we obtain that for every $T > 0$ there exists a constant $C > 0$ such that 
\begin{equation*}
d^2(u_\tau^n, u_\star) \leq C 
\end{equation*}
whenever $n\tau \leq T, \ \tau\leq \frac{1}{B}$. 

\begin{lemma}(A discrete version of Gronwall lemma, \cite{AmbrosioGigliSavare05}) Let $A_1, \alpha \in [0, +\infty)$ and, for $n\geq 1$, let $a_n, \tau_n \in[0, +\infty)$ be satisfying
\begin{equation*}
a_n \leq A_1 + \alpha \sum_{j=1}^{n}{\tau_j a_j} \ \forall n\geq 1, \ m:= \sup_{n\in \N} \alpha \tau_n < 1.
\end{equation*}
Then, setting $\beta = \frac{\alpha}{1-m}, \ A_2 := \frac{A_1}{1-m}$ and $\tau_0 = 0$, we have
\begin{equation*}
a_n \leq A_2 e^{\beta\sum_{i=0}^{n-1}{\tau_i}} \ \forall n\geq 1. 
\end{equation*} 
\end{lemma}

In addition, the following estimates hold
\begin{equation*}
\phi_{\epsilon(\tau)}(u_\tau^n) \ \leq \ \gamma_\tau n \tau + \phi_{\epsilon(\tau)}(u_\tau^0)
\end{equation*}
and
\begin{eqnarray*}
\sum_{j=1}^{n}{\frac{d^2(u_\tau^j, u_\tau^{j-1})}{2\tau´}} &\leq& \phi_{\epsilon(\tau)}(u_\tau^0) - \phi_{\epsilon(\tau)}(u_\tau^n) + \gamma_\tau n \tau \\
&\leq& \phi_{\epsilon(\tau)}(u_\tau^0) + A + Bd^2(u_\tau^n, u_\star) + \gamma_\tau n\tau.
\end{eqnarray*}

We define $|U'_\tau|: [0, +\infty) \rightarrow [0, +\infty)$ by 
\begin{equation*}
|U'_\tau|(t) := \frac{d(u_\tau^j, u_\tau^{j-1})}{\tau} \text{ if } t\in ((j-1)\tau, j\tau].
\end{equation*}
The last estimate shows that there exist $\tau_k \downarrow 0, \ \cA\in L^2_{loc}([0, +\infty))$ such that 
\begin{equation*}
|U'_{\tau_k}| \text{ converges weakly in } L^2(0,T) \text{ to } \cA \text{ for all } T > 0.
\end{equation*}

Now, we can apply the refined version of Ascoli-Arzel\`a theorem stated below to conclude from the preceding estimates and from (\ref{eq: compact}) that there exist a further subsequence again denoted by $(\tau_k)_{k\in \N}$ and a curve $u: [0, +\infty) \rightarrow \SS$ such that
\begin{equation*}
\overline{u}_{\tau_k}(t) \stackrel{\sigma}{\weakto} u(t) \text{ for all } t\geq 0.
\end{equation*}
Doing so we choose
\begin{equation*}
\omega (s,t) := \int_{s}^{t}{\cA (r) \ dr}
\end{equation*}
since for $0\leq s\leq t$ we have 
\begin{equation*}
\mathop{\limsup}_{k\to \infty} d(\overline{u}_{\tau_k}(s), \overline{u}_{\tau_k}(t)) \ \leq \ \mathop{\limsup}_{k\to \infty} \int_{s}^{t+\tau_k}{|U'_{\tau_k}|(r) \ dr} \ = \ \int_{s}^{t}{\cA(r) \ dr}.
\end{equation*}

\begin{lemma}(A refined version of Ascoli-Arzel\`a theorem, \cite{AmbrosioGigliSavare05}) Let $T > 0$, let $K \subset \SS$ be a sequentially compact set w.r.t. $\sigma$, and let $u_n: [0, T] \rightarrow \SS$ be curves such that 
\begin{eqnarray*}
u_n(t) \in K \text{   } \forall n\in \N, \ t\in [0, T], \\
\mathop{\limsup}_{n\to \infty} d(u_n(s), u_n(t)) \leq \omega (s,t) \text{   } \forall s,t \in[0, T], 
\end{eqnarray*}
for a (symmetric) function $\omega: [0, T] \times [0, T] \rightarrow [0, +\infty)$, such that
\begin{equation*}
\mathop{\lim}_{(s,t) \to (r,r)} \omega (s,t) = 0 \ \forall r\in [0, T]\setminus \CC,  
\end{equation*}
where $\CC$ is an (at most) countable subset of $[0, T]$. 

Then there exist $n_k \uparrow +\infty, \ u: [0, T] \rightarrow \SS$ such that 
\begin{equation*}
u_{n_k}(t) \stackrel{\sigma}{\weakto} u(t) \ \forall t\in [0, T], \ u \text{ is } d\text{-continuous in } [0, T]\setminus \CC. 
\end{equation*}
\end{lemma}

It remains to prove that $u$ is locally absolutely continuous. Using the $\sigma$-lower semicontinuity of the distance $d$ we obtain for all $0\leq s\leq t$ 
\begin{eqnarray*}
d(u(s),u(t)) \ \leq \ \mathop{\liminf}_{k\to \infty} d(\overline{u}_{\tau_k}(s), \overline{u}_{\tau_k}(t)) \ \leq \ \int_{s}^{t}{\cA(r) \ dr}. 
\end{eqnarray*}
Thus, $u$ is locally absolutely continuous with $|u'| \leq \cA \ \LL^1$-a.e. and 
\begin{equation}\label{eq: derivative of u}
\int_{0}^{t}{|u'|^2(s) \ ds} \ \leq \ \mathop{\liminf}_{k\to \infty} \int_{0}^{t}{|U'_{\tau_k}|^2(s) \ ds} 
\end{equation}
for all $t \geq 0$.\\

\framebox{\textbf{Energy inequality}} We prove (\ref{eq: energy inequality}). Let $t > 0$ be given. 

For $\tau > 0$ fixed, we choose $N_\tau \in \N$ such that $t\in ((N_\tau-1)\tau, N_\tau \tau]$. Then we have
\begin{eqnarray*}
&& \phi_{\epsilon(\tau)}(u_\tau^0) - \phi_{\epsilon(\tau)}(\overline{u}_{\tau}(t)) \ = \ \sum_{j=1}^{N_\tau}{\left(\phi_{\epsilon(\tau)}(u_\tau^{j-1}) - \phi_{\epsilon(\tau)}(u_\tau^j)\right)} \\
&\geq& \sum_{j=1}^{N_\tau}{\left(\phi_{\epsilon(\tau)}(u_\tau^{j-1}) - \cY_\tau \phi_{\epsilon(\tau)}(u_\tau^{j-1}) - \gamma_\tau \tau \right)} \ + \ \sum_{j=1}^{N_\tau}{\frac{1}{2\tau}d^2(u_\tau^j, u_\tau^{j-1})} \\
&=& \sum_{j=1}^{N_\tau}{\left(\int_{0}^{\tau}{\frac{\phi_{\epsilon(\tau)}(u_\tau^{j-1}) - \cY_\tau \phi_{\epsilon(\tau)}(u_\tau^{j-1})}{\tau} \ d\eta} +  \frac{1}{2}\int_{(j-1)\tau}^{j\tau}{|U'_\tau|^2(r) \ dr}  -  \gamma_\tau \tau \right)}. 
\end{eqnarray*}
Since for $j\geq 2$ it holds that
\begin{equation*}
u_\tau^{j-1} \ = \ \overline{u}_{\tau}(\underbrace{(j-2)\tau + \eta}_{=: s}) \text{ for } 0 < \eta < \tau,
\end{equation*}
we obtain 
\begin{equation*}
\int_{0}^{\tau}{\frac{\phi_{\epsilon(\tau)}(u_\tau^{j-1}) - \cY_\tau \phi_{\epsilon(\tau)}(u_\tau^{j-1})}{\tau} \ d\eta} \ = \ \int_{(j-2)\tau}^{(j-1)\tau}{\frac{\phi_{\epsilon(\tau)}(\overline{u}_\tau(s)) - \cY_\tau \phi_{\epsilon(\tau)}(\overline{u}_\tau(s))}{\tau} \ ds}
\end{equation*} 
for $2 \leq j \leq N_\tau$.\\

Thus, the following inequality holds
\begin{eqnarray*}
\phi_{\epsilon(\tau)}(u_\tau^0) - \phi_{\epsilon(\tau)}(\overline{u}_\tau(t)) \ \geq \ \int_{0}^{(N_\tau-1)\tau}{\frac{\phi_{\epsilon(\tau)}(\overline{u}_\tau(s)) - \cY_\tau \phi_{\epsilon(\tau)}(\overline{u}_\tau(s))}{\tau} \ ds} \\ + \ \frac{1}{2}\int_{0}^{t}{|U'_\tau|^2(s) \ ds} \ - \ \gamma_\tau N_\tau \tau.
\end{eqnarray*}
 
Now, let $\tau_k \downarrow 0, \ u: [0, +\infty) \rightarrow \SS$ satisfy (\ref{eq: convergence}) (for details see the first part of the proof). Assumption \ref{ass: liminf} and (\ref{eq: recovery sequence}) imply that 
\begin{equation*}
\phi(u(0)) - \phi(u(t)) \ \geq \ \mathop{\limsup}_{k\to \infty} \left(\phi_{\epsilon(\tau_k)}(u_{\tau_k}^{0}) - \phi_{\epsilon(\tau_k)}(\overline{u}_{\tau_k}(t))\right). 
\end{equation*}
All in all we obtain
\begin{eqnarray*}
\phi(u(0)) - \phi(u(t)) \ \geq \ \mathop{\liminf}_{k\to\infty} \int_{0}^{(N_{\tau_k}-1)\tau_k}{\frac{\phi_{\epsilon(\tau_k)}(\overline{u}_{\tau_k}(s)) - \cY_{\tau_k} \phi_{\epsilon(\tau_k)}(\overline{u}_{\tau_k}(s))}{\tau_k} \ ds} \\ + \ \mathop{\liminf}_{k\to\infty} \frac{1}{2}\int_{0}^{t}{|U'_{\tau_k}|^2(s) \ ds}. 
\end{eqnarray*}
We used the fact that $N_\tau \tau \to t$ and thus $\gamma_\tau N_\tau \tau \to 0$ $(\tau\to 0)$.\\ 

The energy inequality (\ref{eq: energy inequality})
\begin{equation*}
\phi(u(0)) - \phi(u(t)) \ \geq \ \frac{1}{2} \int^{t}_{0}{|\partial^- \phi|^2 (u(s)) \ ds} \ + \ \frac{1}{2} \int^{t}_{0}{|u'|^2 (s) \ ds}
\end{equation*}
follows by applying (\ref{eq: derivative of u}), Fatou's Lemma and assumption \ref{ass: crucial assumption}. \\

\framebox{\textbf{$u$ is a curve of maximal slope}} 
Let $u$ be a locally absolutely continuous curve satisfying (\ref{eq: convergence}) - (\ref{eq: energy inequality}) and let $|\partial^- \phi|$ be a strong upper gradient for $\phi$. Then, by definition, for all $t > 0$ 
it holds that
\begin{eqnarray*}
\phi(u(0)) - \phi(u(t)) &\leq& \int_{0}^{t}{|\partial^- \phi|(u(r)) |u'|(r) \ dr} \\
&\leq& \frac{1}{2} \int_{0}^{t}{|\partial^- \phi|^2(u(r)) \ dr} \ + \ \frac{1}{2} \int_{0}^{t}{|u'|^2(r) \ dr}.
\end{eqnarray*}
Hence, since by (\ref{eq: energy inequality}) also the reverse inequality holds, we obtain 
\begin{equation}\label{eq: energy equality}
\phi(u(s)) - \phi(u(t)) \ = \ \frac{1}{2} \int_{s}^{t}{|\partial^- \phi|^2(u(r)) \ dr} \ + \ \frac{1}{2} \int_{s}^{t}{|u'|^2(r) \ dr}
\end{equation}
for all $0 \leq s \leq t$.\\ 

The proof of Theorem \ref{thm: main theorem} is complete.
\end{proof} 

\subsection{Supplement to Theorem \ref{thm: main theorem}}\label{subsec: supplement to the theorem}

\paragraph{Convergence condition}

The following convergence condition (\ref{eq: convergence condition}) in case $|\partial^- \phi|$ is only a weak upper gradient for $\phi$ is the counterpart of the continuity condition (\ref{eq: continuity condition}) in section \ref{subsec: MM}: 
\begin{equation}\label{eq: convergence condition}
v_{n}\stackrel{\sigma}{\weakto}v, \ \sup_{n\in\N}\left\{\chi_{\epsilon_n, \tau_n}(v_{n}), d(v_{n},v), \phi_{\epsilon_n}(v_{n})\right\} < +\infty \ \Rightarrow \ \phi_{\epsilon_n}(v_{n})\to \phi(v)  
\end{equation}
for $\tau_n \to 0 \ (\tau_n > 0)$, with $\chi_{\epsilon, \tau}(\cdot) := \frac{\phi_{\epsilon}(\cdot) - \cY_\tau \phi_\epsilon(\cdot)}{\tau}$ and $\epsilon_n := \epsilon(\tau_n)$.

\begin{proposition}\label{prop: supplement to the theorem}
Let the assumptions \ref{ass: co}, \ref{ass: liminf} and \ref{ass: crucial assumption}, and in addition the convergence condition (\ref{eq: convergence condition}) be satisfied. Let $u: [0, +\infty)\rightarrow \SS$ be constructed according to (\ref{eq: convergence}) in Theorem \ref{thm: main theorem}.  
Then the energy dissipation inequality
\begin{equation}\label{eq: refined energy inequality}
\phi(u(s)) - \phi(u(t)) \ \geq \ \frac{1}{2} \int^{t}_{s}{|\partial^- \phi|^2 (u(r)) \ dr} \ + \ \frac{1}{2} \int^{t}_{s}{|u'|^2 (r) \ dr} 
\end{equation}
holds for all $t \geq 0$ and a.e. $s\in (0, t)$.

In particular, if $|\partial^-\phi|$ is a weak upper gradient for $\phi$, then $u$ is a curve of maximal slope for $\phi$ with respect to $|\partial^-\phi|$. 
\end{proposition}

\begin{proof}
Note that the calculations in the second step of the proof of Theorem \ref{thm: main theorem} show that for every $t > 0$ and for $\LL^1$-a.e. $s\in (0,t)$ there exists a further subsequence $(\overline{u}_{\tau_{k_l}})_{l\in \N}$ such that
\begin{equation*}
\sup_{l\in \N} \chi_{\epsilon(\tau_{k_l}), \tau_{k_l}}(\overline{u}_{\tau_{k_l}}(s)) < +\infty.
\end{equation*} 
Then, by using (\ref{eq: convergence condition}) and the same arguments as in the proof of the energy inequality (\ref{eq: energy inequality}), we can conclude.
\end{proof}

\paragraph{Generalization to a non-uniform distribution of the error}

The proof of Theorem \ref{thm: main theorem} shows that we may replace the error $\gamma_\tau\tau$ in the approximate minimization problem (\ref{eq: relaxed form of minimi}) by a more general $\gamma_\tau^{(n)}$ depending on $n\in\N$. Indeed, the convergence 
\begin{equation}
\sum_{j=1}^{N_\tau}{\gamma_\tau^{(j)}}\to 0 \text{ as } \tau\to 0
\end{equation} 
for $N_\tau\in \N$ as in the second step of the proof, is sufficient for our purposes.\\   

Hence, we can extend our theory to a non-uniform distribution of the error and to the following relaxed minimizing movement scheme along $(\phi_\epsilon)_{\epsilon > 0}$:\\
 
The sequence of initial values $u_\tau^0 \stackrel{\sigma}{\weakto} u^0 \in D(\phi) \ (\tau\to 0)$ satisfies
\begin{equation}\label{eq: non-uniform recovery sequence}
\sup_\tau d(u_\tau^0, u^0) < +\infty, \ \phi_{\epsilon(\tau)}(u_\tau^0) \to \phi(u^0)
\end{equation}
and $u_\tau^n \ (\tau > 0, n\in \N)$ satisfies
\begin{equation}\label{eq: non-uniform relaxed form of minimi}
\phi_{\epsilon(\tau)}(u_\tau^n) + \frac{1}{2\tau} d^2(u_\tau^n, u_\tau^{n-1}) \ \leq \ \inf_{v\in \SS} \left\{\phi_{\epsilon(\tau)}(v) + \frac{1}{2\tau} d^2(v,u_\tau^{n-1}) \right\} + \gamma_\tau^{(n)}
\end{equation}
with some $(\gamma_\tau^{(n)})_{n\in\N}$, $\gamma_\tau^{(n)} > 0$ and $\sum_{j=1}^{N_\tau}{\gamma_\tau^{(j)}} \to 0$ as $\tau\to0$ whenever $(N_\tau\tau)_{\tau > 0}$ is bounded.\\

For the sake of clear presentation, we come back to considering a uniform distribution of the error in the remaining part of the paper. 
Note that all the results in this paper can be obtained for a non-uniform distribution of the error as well.

\newpage
\section{The case $\phi_\epsilon = \phi_{\tilde{\epsilon}}$ for all $\epsilon, \tilde{\epsilon} > 0$}\label{sec: constant sequences}

The special case in which we have $\phi_\epsilon = \phi_{\tilde{\epsilon}}$ for all $\epsilon, \tilde{\epsilon} > 0$ is worth considering in order to better understand Theorem \ref{thm: main theorem} and our main assumption \ref{ass: crucial assumption}. Moreover, we will obtain interesting relaxation results. 

\subsection{Assumption \ref{ass: crucial assumption} with regard to a single functional}\label{subsec: An important proposition}

We give a positive answer to the question if there is an interrelation corresponding to assumption \ref{ass: crucial assumption} between the (relaxed) minimizing movement scheme along a single functional $\psi: \SS \rightarrow (-\infty, +\infty]$ and its relaxed slope . 

\begin{proposition}\label{prop: important proposition}
Let $\psi: \SS \rightarrow (-\infty, +\infty]$, $D(\psi) \neq \emptyset,$ satisfy the following assumptions: 
\begin{equation}
\psi \text{ is } d\text{-lower semicontinuous}, 
\end{equation}
and for small $\tau > 0$ it holds that
\begin{equation}\label{eq: existence of minimizers}
J_\tau \psi [u] \neq \emptyset \text{ for all } u\in\SS.
\end{equation}
Then 
whenever $u, u_\tau\in \SS \ (\tau > 0)$ satisfy 
\begin{equation*}
u_\tau \stackrel{\sigma}{\weakto} u, \ \sup_\tau \left\{\psi(u_\tau), d(u_\tau, u)\right\} < +\infty
\end{equation*}
it holds that
\begin{equation}\label{ass: crucial assumption for a single functional}
\mathop{\liminf}_{\tau \to 0} \frac{\psi(u_\tau) - \cY_\tau \psi(u_\tau)}{\tau} \ \geq \ \frac{1}{2} |\partial^- \psi|^2 (u).
\end{equation} 
\end{proposition} 

\begin{proof}
Note that (\ref{eq: existence of minimizers}) implies that there exist $A, B > 0, \ u_\star \in \SS$ such that 
\begin{equation*}
\psi(\cdot) \geq -A -Bd^2(\cdot, u_\star).
\end{equation*}
Let $u, u_\tau\in \SS \ (\tau > 0)$ satisfy 
\begin{equation*}
u_\tau \stackrel{\sigma}{\weakto} u, \ \sup_\tau \left\{\psi(u_\tau), d(u_\tau, u)\right\} < +\infty. 
\end{equation*}
For $\pi \in (0, 1)$ we set $v_{\pi \cdot\tau} \in J_{\pi \cdot \tau} \psi[u_\tau]$. We have
\begin{eqnarray*}
d^2(v_{\pi\cdot\tau}, u_\tau) &\leq& 2\pi \tau(\psi(u_\tau) - \psi(v_{\pi \cdot\tau})) \\
&\leq& 2\pi \tau (\psi(u_\tau) + A + Bd^2(v_{\pi\cdot\tau}, u_\star)) \\
\end{eqnarray*}
leading to 
\begin{eqnarray*}
d^2(v_{\pi \cdot\tau}, u_\tau)[1-2\pi\tau B] - 4\pi\tau Bd(u_\tau, u_\star) d(v_{\pi\cdot\tau}, u_\tau) \\ \leq \ 2\pi \tau \psi(u_\tau) + 2\pi \tau A + 2\pi \tau Bd^2(u_\tau, u_\star).
\end{eqnarray*}
We can conclude that
\begin{equation*}
d(v_{\pi\cdot\tau}, u_\tau) \to 0 \text{ as } \tau \to 0
\end{equation*}
and $(v_{\pi\cdot\tau})_{\tau > 0}$ is an admissible test sequence in the definition of $|\partial^- \psi|(u)$.\\

Now, it holds that
\begin{eqnarray*}
\frac{1}{2}|\partial^- \psi|^2(u) &=& \frac{1}{2}\int_{0}^{1}{|\partial^- \psi|^2(u) \ d\pi} \ \leq \ \frac{1}{2}\int_{0}^{1}{\mathop{\liminf}_{\tau \to 0} |\partial \psi|^2(v_{\pi\cdot\tau}) \ d\pi} \\
&\leq& \frac{1}{2}\int_{0}^{1}{\mathop{\liminf}_{\tau\to 0} \frac{d^2(v_{\pi\cdot\tau}, u_\tau)}{(\pi \cdot \tau)^2} \ d\pi} \ \leq \ \mathop{\liminf}_{\tau\to 0} \frac{1}{2}\int_{0}^{1}{\frac{d^2(v_{\pi\cdot\tau}, u_\tau)}{(\pi \cdot \tau)^2} \ d\pi} \\
&=& \mathop{\liminf}_{\tau\to 0} \frac{1}{\tau} \int_{0}^{\tau}{\frac{d^2(v_\varsigma, u_\tau)}{2\varsigma^2} \ d\varsigma} \ = \ \mathop{\liminf}_{\tau \to 0} \frac{\psi(u_\tau) - \cY_\tau \psi(u_\tau)}{\tau}. 
\end{eqnarray*}
We applied Fatou's Lemma and the following lemma with collected statements from (\cite{AmbrosioGigliSavare05}, chapter 3):
\begin{lemma}(\cite{AmbrosioGigliSavare05})
Let $\psi: \SS \rightarrow (-\infty, +\infty]$ be given. For all $w\in \SS, \ \tau > 0$ and $w_\tau \in J_\tau \psi[w]$ it holds that
\begin{equation}
|\partial \psi|(w_\tau) \ \leq \ \frac{d(w_\tau, w)}{\tau}.
\end{equation}
We set 
\begin{equation*}
d_\tau^+ (w) := \sup_{w_\tau \in J_\tau[w]} d(w_\tau, w), \ d_\tau^- (w) := \inf_{w_\tau \in J_\tau[w]} d(w_\tau, w) \ (w\in \SS, \ \tau > 0). 
\end{equation*}
Now, let $\psi$ be $d$-lower semicontinuous and satisfy (\ref{eq: existence of minimizers}) for all $\tau\in (0, \tau_\star)$ and let $w\in D(\psi)$. Then there exists an (at most) countable set $\NN_w \subset (0, \tau_\star)$ such that
\begin{equation*}
d_\tau^- (w) \ = \ d_\tau^+ (w) \text{ for all } \tau \in (0, \tau_\star) \setminus \NN_w
\end{equation*}
and the map $\tau \mapsto \frac{d_\tau^\pm (w)}{\tau}$ has finite pointwise variation in $(\tau_0, \tau_1)$ for every $0 < \tau_0 < \tau_1 < \tau_\star$.
Moreover, it holds for all $\tau \in (0, \tau_\star)$ and $w_\tau \in J_\tau \psi[w]$ that
\begin{equation}
\frac{d^2(w_\tau, w)}{2\tau} \ + \ \int_{0}^{\tau}{\frac{(d_r^\pm(w))^2}{2r^2} \ dr} \ = \ \psi(w) - \psi(w_\tau). 
\end{equation}

\end{lemma}
The proof of Proposition \ref{prop: important proposition} is complete. 
\end{proof}

\subsection{Relaxation results}\label{subsec: relaxation results}

\paragraph{Relaxed form of minimization}

We come back to the classical minimizing movement scheme for gradient flows (section \ref{subsec: MM}) along a single functional used to construct curves of maximal slope \cite{AmbrosioGigliSavare05}. 
Let us introduce a relaxed form of minimization in each step of the scheme. Still we obtain curves of maximal slope.  
Let $\psi: \SS \rightarrow (-\infty, +\infty]$ be given. \\

For every time step $\tau > 0$, find a sequence $(u_\tau^n)_{n\in\N}$ by the following scheme. 

The sequence of initial values $u_\tau^0 \stackrel{\sigma}{\weakto} u^0 \in D(\psi) \ (\tau\to 0)$ satisfies
\begin{equation}\label{eq: recovery sequence 2}
\sup_\tau d(u_\tau^0, u^0) < +\infty, \ \psi(u_\tau^0) \to \psi(u^0)
\end{equation}
and $u_\tau^n\in\SS \ (n\in \N)$ satisfies
\begin{equation}\label{eq: relaxed form of minimi 2}
\psi(u_\tau^n) + \frac{1}{2\tau} d^2(u_\tau^n, u_\tau^{n-1}) \ \leq \ \inf_{v\in \SS} \left\{\psi(v) + \frac{1}{2\tau} d^2(v,u_\tau^{n-1}) \right\} + \gamma_\tau \tau
\end{equation}
with some $\gamma_\tau > 0, \ \gamma_\tau \to 0$ as $\tau \to 0$. \\

Let $\overline{u}_\tau : [0, +\infty) \rightarrow \SS$ be the corresponding piecewise constant interpolation, i.e. 
\begin{eqnarray}
\overline{u}_\tau (t) &\equiv& u_\tau^n \text{    if } t\in ((n-1)\tau, n\tau], \ n\in \N, \label{eqn: mm 2} \\
\overline{u}_\tau (0) &=& u_\tau^0. \label{eqn: initial value 2}
\end{eqnarray}

Then the following holds. 

\begin{theorem}\label{thm: relaxed form of minimization}
Let $\psi$ satisfy assumption \ref{ass: AGS1} and $\overline{u}_\tau : [0, +\infty) \rightarrow \SS \ (\tau > 0)$ be constructed according to (\ref{eq: recovery sequence 2}) - (\ref{eqn: initial value 2}). Then there exist a locally absolutely continuous curve $u: [0, +\infty) \rightarrow \SS$ and a subsequence $(\overline{u}_{\tau_k})_{k\in \N}$ such that 
\begin{equation}\label{eq: convergence 2}
\overline{u}_{\tau_k} (t) \stackrel{\sigma}{\weakto} u(t) \text{ for all } t \geq 0
\end{equation} 
and $u$ satisfies the initial condition
\begin{equation}\label{eq: initial condition 2}
u(0) = u^0
\end{equation}
and the energy inequality (\ref{eq: AGS energy inequality}) for all $t \geq 0$. In particular, if $|\partial^-\psi|$ is a strong upper gradient for $\psi$, then
\begin{equation}
u \text{ is a curve of maximal slope for } \psi \text{ with respect to } |\partial^- \psi|. 
\end{equation}
\end{theorem}

\begin{proof}
We apply Theorem \ref{thm: main theorem} with $\phi_\epsilon = \phi = \psi$ for all $\epsilon > 0$. Assumption \ref{ass: AGS1} corresponds to the assumptions \ref{ass: co} and \ref{ass: liminf}. It can be proved (\cite{AmbrosioGigliSavare05}, chapter 2) that assumption \ref{ass: AGS1} implies that (\ref{eq: existence of minimizers}) holds for $\tau > 0$ small enough. Thus, Proposition \ref{prop: important proposition} shows that assumption \ref{ass: crucial assumption} is also satisfied in this case.  
\end{proof}

There is some possibility that the relaxed minimizing movement scheme produces more curves of maximal slope than the classical scheme. We give a simple example. 
\begin{example}
We apply our scheme (\ref{eq: recovery sequence 2}) - (\ref{eq: relaxed form of minimi 2}) to $\psi(x) = -\frac{1}{\alpha}|x|^\alpha$ with $1 < \alpha < 2$, $\psi: \R \rightarrow \R$, and initial values $u_\tau^0 = u^0 = 0$ for all $\tau > 0$. Doing so we also obtain the trivial solution $u \equiv 0$ to the corresponding gradient flow equation since the choice $u_\tau^n = 0 \ (\tau > 0, n\in\N)$ is admissible in (\ref{eq: relaxed form of minimi 2}). This solution cannot be reached by the classical scheme.   
\end{example}

The error order $o(\tau)$ in (\ref{eq: relaxed form of minimi}) and (\ref{eq: relaxed form of minimi 2}) is optimal. The following example shows that we cannot expect to obtain curves of maximal slope if we allow an error greater than of order $o(\tau)$. 

\begin{example}
We consider $\psi: \R \rightarrow \R$, $\psi(x) = \frac{1}{2}x^2$ with initial values $u_\tau^0 = u^0 = 0$ for all $\tau > 0$. Let us choose $u_\tau^1 = \tau$ and $u_\tau^n = \frac{n\tau}{(1+\tau)^{n-1}}$ for $n\geq2$. Then we have
\begin{equation*}
\psi(u_\tau^n) + \frac{1}{2\tau} d^2(u_\tau^n, u_\tau^{n-1}) \ \leq \ \cY_\tau \psi (u_\tau^{n-1}) + \tau  
\end{equation*}
for all $0 < \tau < 1$ and $n\in \N$, and the limit curve $u: [0, +\infty) \rightarrow \R$, $u(t) = te^{-t}$ does not solve the corresponding gradient flow equation.
\end{example}

\paragraph{Lower semicontinuous envelope relaxation}

Now, we consider a functional $\psi: \SS \rightarrow (-\infty, +\infty]$ satisfying the coercivity and compactness conditions (\ref{eq: AGS1 coercive}) and (\ref{eq: AGS1 compact}), but for which the minimization problems in the classical minimizing movement scheme for gradient flows (section \ref{subsec: MM}) need not have any solutions due to missing lower semicontinuity (\ref{eq: AGS1 lower semicontinuity}). 

However, the approximate minimization problems (\ref{eq: relaxed form of minimi 2}) are solvable and we can follow the steepest descent movements in our relaxed scheme along the functional $\psi$. They are closely related to the gradient flow motion of the lower semicontinuous envelope $\psi_{sc}: \SS \rightarrow (-\infty, +\infty]$ of $\psi$.\\ 

To be more precise, let $\psi_{sc}: \SS \rightarrow (-\infty, +\infty]$ be both the $d$-lower semicontinuous and a weakened form of the $\sigma$-lower semicontinuous envelope of the functional $\psi$, i.e.
\begin{equation}\label{eq: liminf 3}
\sup_{n,m} d(u_n, u_m) < +\infty, \ u_n \stackrel{\sigma}{\weakto} u \ \Rightarrow \ \mathop{\liminf}_{n\to \infty} \psi(u_n) \geq \psi_{sc}(u)
\end{equation} 
and for all $u\in D(\psi)$ there exists $(v_n)_{n\in\N} \subset \SS$ such that 
\begin{equation}\label{eq: d-lower semicontinuous envelope}
v_n \stackrel{d}{\to} u, \ \psi(v_n) \to \psi_{sc}(u). 
\end{equation} 

Moreover, let initial values $u_\tau^0 \stackrel{\sigma}{\weakto} u^0 \in D(\psi) \ (\tau\to 0)$ be given satisfying
\begin{equation}\label{eq: recovery sequence 3}
\sup_\tau d(u_\tau^0, u^0) < +\infty, \ \psi(u_\tau^0) \to \psi_{sc}(u^0).
\end{equation}

\begin{theorem}\label{thm: relaxation}
We assume that $\psi$ satisfies (\ref{eq: AGS1 coercive}) and (\ref{eq: AGS1 compact}) and that $\psi_{sc}$ is given by (\ref{eq: liminf 3}) and (\ref{eq: d-lower semicontinuous envelope}). We construct $\overline{u}_\tau: [0, +\infty) \rightarrow \SS$ according to (\ref{eq: recovery sequence 3}), (\ref{eq: relaxed form of minimi 2}), (\ref{eqn: mm 2}) and (\ref{eqn: initial value 2}). 

Then there exist a locally absolutely continuous curve $u: [0, +\infty) \rightarrow \SS$ and a subsequence $(\overline{u}_{\tau_k})_{k\in \N}$ such that 
\begin{equation}\label{eq: convergence 3}
\overline{u}_{\tau_k} (t) \stackrel{\sigma}{\weakto} u(t) \text{ for all } t \geq 0
\end{equation} 
and $u$ satisfies the initial condition
\begin{equation}\label{eq: initial condition 3}
u(0) = u^0
\end{equation}
and the energy inequality
\begin{equation}\label{eq: energy inequality 3}
\psi_{sc}(u(0)) - \psi_{sc}(u(t)) \ \geq \ \frac{1}{2} \int^{t}_{0}{|\partial^- \psi_{sc}|^2 (u(s)) \ ds} \ + \ \frac{1}{2} \int^{t}_{0}{|u'|^2 (s) \ ds}
\end{equation}
for all $t\geq 0$.

In particular, if the relaxed slope of $\psi_{sc}$ is a strong upper gradient, then 
\begin{equation}
u \text{ is a curve of maximal slope for } \psi_{sc} \text{ with respect to } |\partial^- \psi_{sc}|. 
\end{equation}
\end{theorem}

\begin{proof}
We apply Theorem \ref{thm: main theorem} with $\phi_\epsilon = \psi$ for all $\epsilon > 0$ and $\phi = \psi_{sc}$. The assumptions \ref{ass: co} and \ref{ass: liminf} are clearly satisfied. It remains to be checked if assumption \ref{ass: crucial assumption} is fulfilled. 

We can conclude from the fact that the conditions (\ref{eq: AGS1 coercive}) and (\ref{eq: AGS1 compact}) hold for the functional $\psi$ and from (\ref{eq: liminf 3}), (\ref{eq: d-lower semicontinuous envelope}) that assumption \ref{ass: AGS1} is satisfied for the functional $\psi_{sc}$. It can be proved \cite{AmbrosioGigliSavare05} that if a functional satisfies assumption \ref{ass: AGS1} then it also satisfies (\ref{eq: existence of minimizers}). So for small $\tau > 0$ we have 
\begin{equation*}
J_\tau \psi_{sc} [w] \neq \emptyset \text{ for all } w\in \SS.
\end{equation*}
Hence, Proposition \ref{prop: important proposition} is applicable to $\psi_{sc}$ and it holds that 
\begin{equation}
\mathop{\liminf}_{\tau \to 0} \frac{\psi_{sc}(u_\tau) - \cY_\tau \psi_{sc}(u_\tau)}{\tau} \ \geq \ \frac{1}{2} |\partial^- \psi_{sc}|^2 (u)
\end{equation}
whenever $u, u_\tau\in \SS \ (\tau > 0)$ satisfy 
\begin{equation*}
u_\tau \stackrel{\sigma}{\weakto} u, \ \sup_\tau \left\{\psi_{sc}(u_\tau), d(u_\tau, u)\right\} < +\infty. 
\end{equation*}

Now, let $u, u_\tau\in \SS \ (\tau > 0)$ with
\begin{equation*}
u_\tau \stackrel{\sigma}{\weakto} u, \ \sup_\tau \left\{\psi(u_\tau), d(u_\tau, u)\right\} < +\infty 
\end{equation*}
be given. For $\tau>0$ fixed, we set $v_\tau \in J_\tau \psi_{sc}[u_\tau]$, we let $(v_\tau^k)_{k\in\N}$ be a recovery sequence for $v_\tau$ according to (\ref{eq: d-lower semicontinuous envelope}), i.e. $v_\tau^k \stackrel{d}{\to} v_\tau, \ \psi(v_\tau^k) \to \psi_{sc}(v_\tau) \ (k\to \infty)$, and we obtain
\begin{eqnarray*}
&& \frac{\psi(u_\tau) - \cY_\tau \psi(u_\tau)}{\tau} \ + \ \frac{\cY_\tau \psi_{sc}(u_\tau) - \psi_{sc}(u_\tau)}{\tau} \ \geq \ \frac{\cY_\tau \psi_{sc}(u_\tau) - \cY_\tau \psi(u_\tau)}{\tau} \\
&\geq& \frac{\psi_{sc}(v_\tau) - \psi_{sc}(v_\tau^k)}{\tau} \ + \ \frac{1}{2\tau} \left[d^2(v_\tau, u_\tau) - d^2(v_\tau^k, u_\tau) \right] \ \geq \ -\tau.
\end{eqnarray*} 
The last inequality holds if we choose $k\in \N$ large enough for $\tau > 0$ fixed. Please note that we have $\psi_{sc}(w) \leq \psi(w)$ for all $w\in\SS$. \\

All in all we obtain
\begin{eqnarray*}
\frac{1}{2}|\partial^- \psi_{sc}|^2(u) &\leq& \mathop{\liminf}_{\tau \to 0} \frac{\psi_{sc}(u_\tau) - \cY_\tau \psi_{sc}(u_\tau)}{\tau} \\
&\leq& \mathop{\liminf}_{\tau \to 0} \frac{\psi(u_\tau) - \cY_\tau \psi(u_\tau)}{\tau}.
\end{eqnarray*} 

The proof of Theorem \ref{thm: relaxation} is complete. 
\end{proof}

We indirectly revealed that for small $\tau > 0$ we have $\cY_\tau \psi_{sc} (w) = \cY_\tau \psi (w)$ for all $w\in\SS$. This equality can also be shown directly by means of the theory of $\Gamma$-convergence. In fact, we could have proved Theorem \ref{thm: relaxation} as a corollary of Theorem \ref{thm: relaxed form of minimization}. 

\paragraph{Restriction to a dense subset}

Now, let a functional $\psi: \SS \rightarrow (-\infty, +\infty]$ satisfying assumption \ref{ass: AGS1} with strong upper gradient $|\partial^- \psi|$ and a dense subset $\VV \subset \SS$ such that
\begin{equation*}
\forall v\in D(\psi) \ \exists (v_k)_{k\in\N}\subset \VV: \ v_k\stackrel{d}{\to} v, \ \psi(v_k)\to \psi(v)
\end{equation*}
be given.
Our relaxation results make it possible to restrict the approximate minimization problems to this dense subset $\VV\subset \SS$, i.e. to replace (\ref{eq: relaxed form of minimi 2}) by 
\begin{equation}
(\psi + \I_\VV) (u_\tau^n) + \frac{1}{2\tau} d^2(u_\tau^n, u_\tau^{n-1}) \ \leq \ \cY_\tau (\psi + \I_\VV) (u_\tau^{n-1}) + \gamma_\tau \tau
\end{equation}
in which we set $\I_\VV \equiv 0$ on $\VV$, $\I_\VV \equiv +\infty$ on $\SS \setminus \VV$. Thus, we construct approximations $\overline{u}_\tau: [0, +\infty) \rightarrow \VV$ (with particular properties) to curves of maximal slope for $\psi$ with respect to $|\partial^-\psi|$. In the following example, the desired property is higher regularity.  
\begin{example}
Let $\Omega \subset \Rn$ be bounded and open. We set
\begin{eqnarray*}
\SS = L^p(\Omega) \ (1\leq p < +\infty), \ \sigma = d, \ D(\psi) = W_0^{1,p}(\Omega), \\
\psi(u) = \int_{\Omega}^{}{L(x, \nabla u(x)) \ dx}  +  \int_{\Omega}^{}{j(x,u(x)) \ dx}, \ \text{ if } u\in W_0^{1,p}(\Omega) 
\end{eqnarray*}
with convex, lower semicontinuous $L(x, \cdot): \Rn \rightarrow \R$, $j(x, \cdot): \R \rightarrow \R$. We suppose that $L(x, \cdot)$ satisfies a growth condition of order $p$, i.e. that there exist $c_1, c_2, c_3, c_4 > 0$ such that
\begin{equation*}
c_1 |z|^p - c_2 \ \leq \ L(x, z) \ \leq \ c_3|z|^p + c_4 
\end{equation*} 
for all $x\in\Rn$, $z\in \R^n$. Moreover, we suppose that $j$ is bounded from below and that there exists $c > 0$ such that
\begin{equation*}
j(x, z) \ \leq \ c(1+|z|^p)
\end{equation*} 
for all $x\in\Rn$, $z\in\R$. 

We set 
\begin{equation*}
\VV = C_0^\infty(\Omega). 
\end{equation*}

One can easily check that all the required assumptions are satisfied in this example.  
\end{example}

\section{Convergence of gradient flows}\label{sec: convergence of gradient flows}

Let us suppose that we have functionals $\phi_\epsilon: \SS \rightarrow (-\infty, +\infty]$ $(\epsilon > 0)$, $\phi: \SS \rightarrow (-\infty, +\infty]$ connected through a $\Gamma(\sigma)$-liminf inequality as in assumption \ref{ass: liminf}, with strong upper gradients $|\partial^- \phi_\epsilon|$, $|\partial^-\phi|$, and satisfying a particular property which we call COGF-Property (Convergence Of Gradient Flows):

\paragraph{COGF-Property}
If $\vartheta_\epsilon: [0, +\infty) \rightarrow \SS$ is a curve of maximal slope for $\phi_\epsilon$ with respect to $|\partial^- \phi_\epsilon|$ $(\epsilon > 0)$ and we have
\begin{eqnarray}
\vartheta_\epsilon(t) \stackrel{\sigma}{\weakto} \vartheta(t) \text{ for all } t\geq 0, \label{eqn: convergence of gradient flows} \\
\phi_{\epsilon}(\vartheta_\epsilon(0)) \to \phi(\vartheta(0)), \label{eqn: well-preparedness of initial values}
\end{eqnarray}
for a curve $\vartheta: [0, +\infty) \rightarrow \SS$, then $\vartheta$ is a curve of maximal slope for $\phi$ with respect to $|\partial^- \phi|$. \\

In addition, we suppose that there exist curves of maximal slope for $\phi_\epsilon$ with respect to $|\partial^- \phi_\epsilon|$, for which (\ref{eqn: convergence of gradient flows}), (\ref{eqn: well-preparedness of initial values}) indeed hold. For this to be guaranteed, a natural setting would include both some equi-coercivity and combined compactness property such as in assumption \ref{ass: co} and some lower semicontinuity and compactness property of the single functionals $\phi_\epsilon$ such as (\ref{eq: AGS1 lower semicontinuity}) and (\ref{eq: AGS1 compact}).  

\paragraph{Our expectation}
Intuitively, in this case, any choice $\epsilon = \epsilon(\tau)$ in (\ref{eq: recovery sequence}) - (\ref{eq: relaxed form of minimi}) should be feasible in order to obtain the results of Theorem \ref{thm: main theorem}.\\

We specify this idea in the abstract situation that the COGF-Property is a consequence of the established considerations by Sandier and Serfaty \cite{SandierSerfaty04}, \cite{serfaty2011gamma}. Indeed, our expectation is fulfilled:  
 
\begin{proposition}\label{prop: Sandier Serfaty}
Let $(\phi_{\epsilon})_{\epsilon > 0}$ satisfy the equi-coercivity condition (\ref{eq: coercive}) and let the lower semicontinuity and compactness conditions (\ref{eq: AGS1 lower semicontinuity}) and (\ref{eq: AGS1 compact}) hold for the single functionals $\phi_\epsilon$. If the Serfaty-Sandier condition  
\begin{equation}\label{eq: Sandier Serfaty}
u_\epsilon \stackrel{\sigma}{\weakto} u \ \Rightarrow \ \mathop{\liminf}_{\epsilon\to 0} |\partial^- \phi_\epsilon|(u_\epsilon) \geq |\partial^- \phi|(u) 
\end{equation} 
holds, then assumption \ref{ass: crucial assumption} is satisfied for every choice $\epsilon = \epsilon(\tau)\to 0 \ (\tau\to0)$.
\end{proposition}

\begin{proof}
Let $\epsilon(\tau) > 0, \ \epsilon(\tau)\to 0$ and $u, u_\tau\in \SS \ (\tau > 0)$ be given such that 
\begin{equation*}
u_\tau \stackrel{\sigma}{\weakto} u, \ \sup_\tau \left\{\phi_{\epsilon(\tau)}(u_\tau), d(u_\tau, u)\right\} < +\infty.  
\end{equation*}
For $\pi\in (0,1)$ we set $v_{\pi\cdot\tau, \tau} \in J_{\pi\cdot\tau}\phi_{\epsilon(\tau)}[u_\tau]$. Then it holds that
\begin{eqnarray*}
&& \frac{1}{2} |\partial^- \phi|^2 (u) \ = \ \frac{1}{2}\int_{0}^{1}{|\partial^- \phi|^2(u) \ \d\pi} \ \leq \ \frac{1}{2}\int_{0}^{1}{\mathop{\liminf}_{\tau\to0} |\partial \phi_{\epsilon(\tau)}|^2 (v_{\pi \cdot \tau, \tau}) \ d\pi} \\
&\leq& \frac{1}{2}\int_{0}^{1}{\mathop{\liminf}_{\tau\to 0} \frac{d^2(v_{\pi\cdot\tau,\tau}, u_\tau)}{(\pi\cdot\tau)^2} \ d\pi} \ \leq \ \mathop{\liminf}_{\tau\to0} \frac{1}{2}\int_{0}^{1}{\frac{d^2(v_{\pi\cdot\tau, \tau}, u_\tau)}{(\pi\cdot\tau)^2} \ d\pi} \\
&=& \mathop{\liminf}_{\tau\to0} \frac{1}{\tau}\int_{0}^{\tau}{\frac{d^2(v_{\varsigma, \tau}, u_\tau)}{2\varsigma^2} \ d\varsigma} \ = \ \mathop{\liminf}_{\tau\to0} \frac{\phi_{\epsilon(\tau)}(u_\tau) - \cY_\tau \phi_{\epsilon(\tau)}(u_\tau)}{\tau}.
\end{eqnarray*}
We applied (\ref{eq: Sandier Serfaty}) and we followed similar arguments as in the proof of proposition \ref{prop: important proposition}. 
\end{proof}

\paragraph{The $\lambda$-convex case} If $\phi, \phi_\epsilon \ (\epsilon > 0)$ are $\lambda$-convex ($\lambda \in \R$), the $\sigma$-topology coincides with the topology induced by the distance $d$ and $\phi_\epsilon \stackrel{\Gamma}{\to}\phi$, then condition (\ref{eq: Ortner}) on the local slopes can be proved \cite{Ortner05TR}
\begin{equation}\label{eq: Ortner}
u_\epsilon \to u \ \Rightarrow \ \mathop{\liminf}_{\epsilon \to 0} |\partial \phi_{\epsilon}|(u_\epsilon) \geq |\partial \phi|(u).   
\end{equation}   

The $\Gamma$-liminf condition (\ref{eq: Ortner 2}) on the local slopes discussed in \cite{Ortner05TR} can be viewed as discrete counterpart of the Serfaty-Sandier condition (\ref{eq: Sandier Serfaty}). \\

Similarly as Proposition \ref{prop: Sandier Serfaty}, we can prove
\begin{proposition}\label{prop: Ortner}
Let $(\phi_{\epsilon})_{\epsilon > 0}$ satisfy the equi-coercivity condition (\ref{eq: coercive}) and let the lower semicontinuity and compactness conditions (\ref{eq: AGS1 lower semicontinuity}) and (\ref{eq: AGS1 compact}) hold for the single functionals $\phi_\epsilon$. If the following condition on the local slopes holds
\begin{equation}\label{eq: Ortner 2}
u_\epsilon \stackrel{\sigma}\weakto u \ \Rightarrow \ \mathop{\liminf}_{\epsilon \to 0} |\partial \phi_{\epsilon}|(u_\epsilon) \geq |\partial \phi|(u),  
\end{equation}
then assumption \ref{ass: crucial assumption} is satisfied for every choice $\epsilon = \epsilon(\tau)\to 0 \ (\tau\to0)$. 
\end{proposition}

Again, we obtain that in our scheme (\ref{eq: recovery sequence})-(\ref{eq: relaxed form of minimi}) any choice $\epsilon = \epsilon(\tau)$ is feasible in order to apply Theorem \ref{thm: main theorem}. \\

In order to prove Proposition \ref{prop: Sandier Serfaty} and \ref{prop: Ortner}, it would be sufficient to impose the conditions (\ref{eq: Sandier Serfaty}) and (\ref{eq: Ortner 2}) respectively on sequences $u_\epsilon \stackrel{\sigma}\weakto u$ with $\sup_\epsilon \left\{\phi_{\epsilon}(u_\epsilon), d(u_\epsilon, u)\right\} < +\infty$, and weaker conditions than (\ref{eq: AGS1 lower semicontinuity}) and (\ref{eq: AGS1 compact}). 

\paragraph{Note}
In the special case considered in section \ref{sec: convergence of gradient flows}, every choice $\epsilon = \epsilon(\tau)$ is possible. In general, the interrelation between the steepest descent movement along a $\Gamma$-converging sequence of functionals and the gradient flow motion of its limit functional is more involved. 

Various choices of $(\epsilon(\tau))_{\tau > 0}$ lead to different motions not coincident with the gradient flow motion of the limit functional.  
For illustrative purposes, the reader may have a look at (heuristical) comptutations of $J_\tau f_\epsilon[\cdot]$ $(\tau > 0, \epsilon > 0)$ for concrete examples of $\Gamma$-converging functionals $f_\epsilon$ in \cite{braides2012local}. 
 
\section{Existence of a suitable choice $\epsilon = \epsilon(\tau)$}\label{sec: in separable metric spaces}

There always exists a sequence $(\epsilon_\tau)_{\tau > 0}$ ($\epsilon_\tau > 0$) such that the following holds: If we associate $\tau > 0$ with $\epsilon = \epsilon(\tau) \leq \epsilon_\tau$ in the relaxed minimizing movement scheme (\ref{eq: recovery sequence}) - (\ref{eq: relaxed form of minimi}), our Theorem \ref{thm: main theorem} with all its results is applicable. 

\begin{theorem}\label{thm: existence of epsilon(tau)}
We assume that $(\SS, d)$ is a separable, complete metric space and the $\sigma$-topology coincides with the topology induced by the distance $d$. 

Let functionals $\phi, \phi_\epsilon: \SS \rightarrow (-\infty, +\infty] \ (\epsilon > 0)$ be given. We suppose that $(\phi_\epsilon)_{\epsilon > 0}$ satisfies assumption \ref{ass: co} and $\phi_\epsilon \stackrel{\Gamma}{\to} \phi$, i.e.
\begin{eqnarray*}
\mathop{\liminf}_{\epsilon\to0} \phi_{\epsilon}(u_\epsilon) \geq \phi(u) \text{ for all } u, u_\epsilon\in\SS, \ u_\epsilon\stackrel{d}{\to}u, \\
\forall u\in \SS \ \exists \tilde{u}_\epsilon\in\SS: \ \tilde{u}_\epsilon\stackrel{d}{\to}u \text{ and } \phi_{\epsilon}(\tilde{u}_\epsilon) \to \phi(u).  
\end{eqnarray*}

Then there exists a sequence $(\epsilon_\tau)_{\tau > 0}$ with $\epsilon_\tau > 0$ such that our main assumption \ref{ass: crucial assumption} is satisfied for all choices $(\epsilon(\tau))_{\tau > 0}$ with $\epsilon(\tau) \leq \epsilon_\tau$ $(\epsilon(\tau)\to0)$. In particular, if we choose $\epsilon = \epsilon(\tau) \leq \epsilon_\tau$ for $\tau \to 0$ in (\ref{eq: recovery sequence}) - (\ref{eq: relaxed form of minimi}), all the results of Theorem \ref{thm: main theorem} hold.   
\end{theorem} 

\paragraph{Comment on Theorem \ref{thm: existence of epsilon(tau)}} The special feature of Theorem \ref{thm: existence of epsilon(tau)} is that the sequence $(\epsilon_\tau)_{\tau > 0}$ only depends on the velocity of $\Gamma$-convergence of the functionals $\phi_\epsilon$ to the functional $\phi$. We notice that the sequence $(\epsilon_\tau)_{\tau > 0}$ in Theorem \ref{thm: existence of epsilon(tau)} is completely independent of initial values $u_\tau^0, u^0\in\SS$ and of (approximate) minimizers $u_\tau^n$ in the (relaxed) minimizing movement scheme. We establish a direct connection between the gradient flow motion along a $\Gamma$-converging sequence of functionals and the gradient flow motion of its limit functional.

\begin{proof}
We prove Theorem \ref{thm: existence of epsilon(tau)}. Note that the functional $\phi$ satisfies assumption \ref{ass: AGS1} since $(\phi_\epsilon)_{\epsilon > 0}$ satisfies assumption \ref{ass: co} and $\phi_\epsilon \ \Gamma$-converges to $\phi$. \\

On a separable metric space, $\Gamma$-convergence is metrizable (\cite{DalMaso93}, chapter 10) for a certain class of functionals.   
\begin{lemma}\label{lem: Gianni Dal Maso}(see \cite{DalMaso93}) Let us fix a dense subset $\left\{x_i: i\in\N\right\}$ of $\SS$, a sequence $(\kappa_j)_{j\in\N}$ of positive real numbers converging to $0$, and an increasing homeomorphism $\Phi$ between $[0, +\infty]$ and $[0, 1]$. 

Then, for a sequence of functionals $f_n : \SS \rightarrow [0, +\infty]$ $(n\in\N)$ satisfying 
\begin{equation}\label{eq: equi-coercive}
(u_n)_{n\in\N}\subset\SS, \ \sup_{n\in\N} f_n(u_n) < +\infty \ \Rightarrow \ \exists u\in\SS, n_k\uparrow+\infty: \ u_{n_k}\stackrel{d}{\to} u, 
\end{equation} 
it holds that $(f_n)_{n\in\N} \ \Gamma$-converges to a functional $f:\SS \rightarrow [0, +\infty]$ $(f_n \stackrel{\Gamma}{\to} f)$ if and only if we have 
\begin{equation}
\delta(f_n, f) \to 0 \ (n\to +\infty),
\end{equation}
with
\begin{equation}\label{eq: definition of metric}
\delta(f_n, f) \ := \ \sum_{i,j = 1}^{\infty}{2^{-i-j}|\Phi(\cY_{\kappa_j} f_n(x_i)) - \Phi(\cY_{\kappa_j}f(x_i))|}. 
\end{equation}
\end{lemma}
We may assume that $\Phi$ is Lipschitz continuous.\\

Now, let $(\alpha_k)_{k\in\N}$ with $\alpha_k > 0 \ (k\in\N)$ and $\alpha_k \to 0 \ (k\to +\infty)$ be given and for $\epsilon, \tau > 0, \ k\in\N$, define $F_{\epsilon, \tau, \alpha_k}, F_{\tau, \alpha_k}: \SS\rightarrow [0, +\infty]$ by
\begin{eqnarray*}
F_{\epsilon,\tau,\alpha_k} &:=& \frac{\phi_\epsilon - \cY_\tau \phi_\epsilon}{\tau} + \alpha_k (\phi_\epsilon + A +Bd^2(\cdot, u_\star)) + d(\cdot, u_\star), \\
F_{\tau, \alpha_k} &:=& \frac{\phi - \cY_\tau \phi}{\tau} + \alpha_k (\phi + A +Bd^2(\cdot, u_\star)) + d(\cdot, u_\star). 
\end{eqnarray*}
Then, for $\tau \in \left(0,\frac{1}{2B}\right)$ and $\alpha_k > 0$ fixed, it holds that
\begin{equation*}
F_{\epsilon, \tau, \alpha_k} \stackrel{\Gamma}{\to} F_{\tau, \alpha_k} \ (\epsilon\to0). 
\end{equation*}
We used the fact that $\cY_\tau \phi_\epsilon$ converges locally uniformly to $\cY_\tau \phi$ as $\epsilon \to0$, which follows from the Fundamental Theorem of $\Gamma$-convergence on the convergence of minimum problems (see \cite{Braides02}, \cite{DalMaso93}). The Fundamental Theorem is applicable due to assumption \ref{ass: co} on $(\phi_\epsilon)_{\epsilon > 0}$. 

At the end, we will be interested in the limit $\tau\to0$, so in the following, if we consider $\tau > 0$, we tacitly suppose that $\tau\in\left(0,\frac{1}{2B}\right)$.  \\

Let $(\theta_\tau)_{\tau>0}$ with $\theta_\tau > 0$ for $\tau > 0$ and $\theta_\tau\to0$ $(\tau\to0)$ be given.\\

Since $(\phi_\epsilon)_{\epsilon > 0}$ satisfies assumption \ref{ass: co}, condition (\ref{eq: equi-coercive}) holds for $(F_{\epsilon, \tau, \alpha_k})_{\epsilon > 0}$. We apply Lemma \ref{lem: Gianni Dal Maso} and deduce that 
\begin{equation*}
\delta(F_{\epsilon, \tau, \alpha_k}, F_{\tau, \alpha_k}) \to 0 \ (\epsilon\to0)
\end{equation*}
for all $\tau, \alpha_k > 0$. Thus, for all $\tau, \alpha_k > 0$ there exists $\epsilon_{\tau, \alpha_k} > 0$ such that
\begin{equation}\label{eq: existence of epsilon(tau, alpha_k)}
\delta(F_{\eta, \tau, \alpha_k}, F_{\tau, \alpha_k}) \leq \theta_\tau \text{ for all } 0 < \eta \leq \epsilon_{\tau, \alpha_k}. 
\end{equation} 

Next, we show that we can choose $\epsilon_\tau = \epsilon_{\tau, \alpha_k}$ independent of $\alpha_k \ (k\in\N)$ in (\ref{eq: existence of epsilon(tau, alpha_k)}). For fixed $\tau > 0$ and for $\epsilon_k\to0 \ (\epsilon_k >0)$ we have
\begin{eqnarray*}
F_{\epsilon_k, \tau, \alpha_k} &\stackrel{\Gamma}{\to}& \frac{\phi - \cY_\tau\phi}{\tau} + d(\cdot, u_\star) \ (k\to +\infty), \\
F_{\tau, \alpha_k} &\stackrel{\Gamma}{\to}& \frac{\phi - \cY_\tau\phi}{\tau} + d(\cdot, u_\star) \ (k\to +\infty). 
\end{eqnarray*}
This means that the two families of functionals $\left\{F_{\epsilon, \tau, \alpha_k}: \epsilon > 0, \ k\in\N \right\}$ and $\left\{F_{\tau, \alpha_k}: k\in\N\right\}$ are uniformly $\Gamma$-equivalent. The notion of $\Gamma$-equivalence was introduced by Braides and Truskinovsky \cite{braides2008asymptotic}.

Moreover, both $(F_{\epsilon_k, \tau, \alpha_k})_{k\in\N}$ and $(F_{\tau, \alpha_k})_{k\in\N}$ satisfy condition (\ref{eq: equi-coercive}). It easily follows (see \cite{braides2008asymptotic} for the corresponding statement) that
\begin{equation}\label{eq: 1}
\sup_{k\in\N} |\cY_{\kappa}F_{\epsilon, \tau, \alpha_k}(x) - \cY_{\kappa}F_{\tau, \alpha_k}(x)| \to 0 \ (\epsilon\to 0)
\end{equation}
for all $\kappa > 0, x\in\SS$.

In view of definition (\ref{eq: definition of metric}) of $\delta(\cdot, \cdot)$ with $\Phi$ Lipschitz continuous, we can conclude from (\ref{eq: 1}) with basic straightforward arguments that for all $\tau > 0$ there exists $\epsilon_\tau > 0$ such that
\begin{equation}\label{eq: existence of epsilon(tau)}
\delta(F_{\eta, \tau, \alpha_k}, F_{\tau, \alpha_k}) \leq \theta_\tau \text{ for all } 0 < \eta \leq \epsilon_\tau, \ \alpha_k > 0. 
\end{equation} 

We want to pass on to the limit $\tau\to0$. \\

Let us consider $(F_{\tau, \alpha_k})_{\tau > 0}$ for $\alpha_k > 0$ fixed.\\

On a separable metric space, $\Gamma$-convergence is compact (see \cite{Braides02}, \cite{DalMaso93}). Thus, every subsequence of $(F_{\tau, \alpha_k})_{\tau > 0}$ admits a $\Gamma$-converging (sub)subsequence. \\ 

If we have 
\begin{equation*}
F_{\tau_n, \alpha_k} \stackrel{\Gamma}{\to} \Sigma \ (n\to+\infty)
\end{equation*} 
for some functional $\Sigma: \SS \rightarrow [0, +\infty]$ and a subsequence $\tau_n\to0$, then the following facts hold.

As $(F_{\tau, \alpha_k})_{\tau > 0}$ satisfies condition (\ref{eq: equi-coercive}), we can apply Lemma \ref{lem: Gianni Dal Maso} to deduce that
\begin{equation}\label{eq: convergence in metric to Sigma}
\delta(F_{\tau_n, \alpha_k}, \Sigma) \to 0 \ (n\to+\infty). 
\end{equation}    

It can be proved \cite{AmbrosioGigliSavare05} that if a functional satisfies assumption \ref{ass: AGS1}, then it also satisfies (\ref{eq: existence of minimizers}). Hence, Proposition \ref{prop: important proposition} is applicable to $\phi$. By the definition of $\Gamma$-convergence it holds that 
\begin{equation*}
\Sigma(u) = \inf \left\{\mathop{\liminf}_{n\to+\infty} F_{\tau_n, \alpha_k}(u_{\tau_n}): \ u_{\tau_n}\stackrel{d}{\to}u\right\}.  
\end{equation*} 
We can conclude that 
\begin{equation}\label{eq: Sigma}
\Sigma(u) \geq \frac{1}{2} |\partial^- \phi|^2(u) + \alpha_k(\phi(u) + A + Bd^2(u, u_\star)) + d(u, u_\star)
\end{equation}
for $u\in\SS$. \\

Now, we can put together all the building blocks of our proof.\\

Let $(\epsilon(\tau))_{\tau > 0}$ with $\epsilon(\tau) \leq \epsilon_\tau$ and $\epsilon(\tau)\to 0$ $(\tau\to0)$ be given with $(\epsilon_\tau)_{\tau > 0}$ as in (\ref{eq: existence of epsilon(tau)}). \\

We consider $(F_{\epsilon(\tau), \tau, \alpha_k})_{\tau > 0}$ for an arbitrary but fixed $\alpha_k > 0$. \\

Let a subsequence $\tau_n \to0$ and a functional $\Sigma: \SS\rightarrow [0,+\infty]$ be given with 
\begin{equation*}
F_{\tau_n, \alpha_k} \stackrel{\Gamma}{\to} \Sigma \ (n\to+\infty), 
\end{equation*}
i.e. as in the preceding considerations on $(F_{\tau, \alpha_k})_{\tau > 0}$. \\

In this case we also have
\begin{equation*}
F_{\epsilon(\tau_n), \tau_n, \alpha_k} \stackrel{\Gamma}{\to} \Sigma \ (n\to+\infty).
\end{equation*}
In fact, this follows from (\ref{eq: existence of epsilon(tau)}), (\ref{eq: convergence in metric to Sigma}), the triangle inequality
\begin{equation*}
\delta(F_{\epsilon(\tau_n), \tau_n, \alpha_k}, \Sigma) \leq \delta(F_{\epsilon(\tau_n), \tau_n, \alpha_k}, F_{\tau_n, \alpha_k}) + \delta(F_{\tau_n, \alpha_k}, \Sigma).
\end{equation*}
and Lemma \ref{lem: Gianni Dal Maso} (note that condition (\ref{eq: equi-coercive}) is satisfied for $(F_{\epsilon(\tau_n), \tau_n, \alpha_k})_{n\in\N}$).\\ 

All in all, remembering the compactness of $\Gamma$-convergence on a separable metric space, we obtain
\begin{equation*}
\mathop{\liminf}_{\tau\to0} F_{\epsilon(\tau), \tau, \alpha_k}(u_\tau) \geq \frac{1}{2} |\partial^- \phi|^2(u) + \alpha_k(\phi(u) + A + Bd^2(u, u_\star)) + d(u, u_\star)
\end{equation*} 
for all $u, u_\tau\in\SS \ (\tau > 0)$ with $u_\tau \stackrel{d}{\to}u$.\\ 

Now, we can directly deduce that assumption \ref{ass: crucial assumption} holds.\\

Let $u, u_\tau \in \SS \ (\tau > 0)$ be given with 
\begin{equation*}
u_\tau \stackrel{d}{\to}u, \ \sup_{\tau} \phi_{\epsilon(\tau)}(u_\tau) < +\infty,  
\end{equation*}
and set $C:= \sup_{\tau} \left(\phi_{\epsilon(\tau)}(u_\tau) + A + Bd^2(u_\tau, u_\star)\right) < +\infty$. 
It follows from the preceding steps that 
\begin{equation}
\frac{1}{2}|\partial^- \phi|^2(u) \ \leq \ \mathop{\liminf}_{\tau\to0} \frac{\phi_{\epsilon(\tau)}(u_\tau) - \cY_\tau \phi_{\epsilon(\tau)}(u_\tau)}{\tau} + \alpha_k C
\end{equation}
for all $\alpha_k > 0$. \\

The proof of assumption \ref{ass: crucial assumption}, and thus, of Theorem \ref{thm: existence of epsilon(tau)} is complete.  
\end{proof}
\newpage
\section{Two finite dimensional examples}\label{sec: 1D examples}

Let a $C^1$-function $\phi: \R \rightarrow \R$ be given, with
\begin{equation*}
\phi(\cdot)  \geq -A -Bd^2(\cdot, u_\star)
\end{equation*} 
for some $A, B > 0, \ u_\star\in\R$. (We write $d(x,y) := |x-y|$.) 

Then assumption \ref{ass: AGS1} is satisfied and the relaxed slope $|\partial^- \phi| = |\phi'|$ is a strong upper gradient for $\phi$. It can be proved \cite{AmbrosioGigliSavare05} that if a functional satisfies assumption \ref{ass: AGS1}, then it also satisfies (\ref{eq: existence of minimizers}). We consider two examples of perturbations $\phi_\epsilon: \R \rightarrow \R$ of $\phi$ with $\phi_\epsilon \stackrel{\Gamma}{\to} \phi$.

\begin{example}\label{ex: oscillation example}
Let an arbitrary function $a: \R \rightarrow [0, +\infty)$ be given and define
\begin{equation*}
\phi_\epsilon (x) = \phi(x) + a(x) \cos^2\left(\frac{x}{\epsilon}\right). 
\end{equation*}
If we choose $\epsilon = \epsilon(\tau)$ with $\frac{\epsilon(\tau)}{\tau} \to 0 \ (\tau\to 0)$ in our scheme (\ref{eq: recovery sequence}) - (\ref{eq: relaxed form of minimi}), then we can apply Theorem \ref{thm: main theorem} and any curve $u$ constructed in accordance with (\ref{eq: convergence}) solves
\begin{eqnarray*}
u'(t) = -\phi'(u(t)) \text{ for all } t\in(0, +\infty), \ 
u(0) = u^0. 
\end{eqnarray*}
\begin{proof}
We only have to check assumption \ref{ass: crucial assumption}. Let $u, u_\tau\in \R \ (\tau > 0)$ be given such that 
\begin{equation*}
u_\tau \to u, \ \sup_\tau \phi_{\epsilon(\tau)}(u_\tau) < +\infty.  
\end{equation*}  
By proposition \ref{prop: important proposition}, it is sufficient to show 
\begin{equation}\label{eq: sufficient condition}
\mathop{\liminf}_{\tau\to0} \frac{\phi_{\epsilon(\tau)}(u_\tau) - \cY_\tau \phi_{\epsilon(\tau)}(u_\tau)}{\tau} \ \geq \ \mathop{\liminf}_{\tau\to 0} \frac{\phi(u_\tau) - \cY_\tau\phi(u_\tau)}{\tau}. 
\end{equation}
We set $v_\tau \in J_\tau\phi[u_\tau]$ and $\tilde{v}_\tau\in \R$ such that
\begin{equation*}
\cos \left(\frac{\tilde{v}_\tau}{\epsilon(\tau)}\right) = 0 \text{ and } d(\tilde{v}_\tau, v_\tau) \leq \pi\epsilon(\tau). 
\end{equation*} 
As in the proof of proposition \ref{prop: important proposition}, we have $v_\tau \to u$ $(\tau\to0)$.

The $C^1$-function $\phi$ is locally Lipschitz continuous and a Lipschitz constant for a neighbourhood of $u$ is denoted by $C(u) > 0$. We use the fact that
\begin{eqnarray*}
\frac{1}{2\tau}d^2(u_\tau, v_\tau)\ \leq \ \phi(u_\tau) - \phi(v_\tau) \ \leq \ C(u)d(u_\tau, v_\tau).  
\end{eqnarray*}

All in all, we obtain

\begin{eqnarray*}
&& \frac{\phi_{\epsilon(\tau)}(u_\tau) - \cY_\tau \phi_{\epsilon(\tau)}(u_\tau)}{\tau} \ + \ \frac{\cY_\tau \phi(u_\tau) - \phi(u_\tau)}{\tau} \ \geq \ \frac{\cY_\tau\phi(u_\tau) - \cY_\tau \phi_{\epsilon(\tau)}(u_\tau)}{\tau} \\
&\geq& \frac{\phi(v_\tau) - \phi(\tilde{v}_\tau)}{\tau} + \frac{1}{2\tau^2}\left[d^2(u_\tau, v_\tau) - d^2(u_\tau, \tilde{v}_\tau)\right]\\
&\geq& -C(u) \frac{d(v_\tau, \tilde{v}_\tau)}{\tau} + \frac{1}{2\tau^2}\left[-d^2(v_\tau, \tilde{v}_\tau) - 2d(u_\tau, v_\tau)d(v_\tau, \tilde{v}_\tau)\right] \\
&\geq& -C(u) \pi \frac{\epsilon(\tau)}{\tau} - \frac{(\pi \epsilon(\tau))^2}{2\tau^2} - 2C(u)\frac{\pi\epsilon(\tau)}{\tau}.  
\end{eqnarray*}
The proof is finished.  
\end{proof}
\end{example} 

\begin{example}\label{ex: finite dimensional perturbation}
Let functions $\zeta_\epsilon: \R \rightarrow \R$ be given with 
\begin{equation*}
\zeta_\epsilon(\cdot) \geq -\tilde{A} - \tilde{B}d^2(\cdot, \tilde{u}_\star)
\end{equation*}
for some $\tilde{A}, \tilde{B} > 0, \ \tilde{u}_\star\in \R$ and such that $(\zeta_\epsilon)_{\epsilon > 0}$ converges locally uniformly to a continuous function $\zeta: \R \rightarrow \R$. We define
\begin{equation*}
\phi_\epsilon(x) = \phi(x) + \epsilon \zeta_\epsilon(x). 
\end{equation*}
If we choose $\epsilon = \epsilon(\tau)$ with $\frac{\epsilon(\tau)}{\tau} \leq C$ for some $C > 0$ in (\ref{eq: recovery sequence}) - (\ref{eq: relaxed form of minimi}), then we can apply Theorem \ref{thm: main theorem} and any curve $u$ constructed in accordance with (\ref{eq: convergence}) solves
\begin{eqnarray*}
u'(t) = -\phi'(u(t)) \text{ for all } t\in(0, +\infty), \ 
u(0) = u^0. 
\end{eqnarray*}
\begin{proof}
We only have to check assumption \ref{ass: crucial assumption}. Let $u, u_\tau\in \R \ (\tau > 0)$ be given such that 
\begin{equation*}
u_\tau \to u, \ \sup_\tau \phi_{\epsilon(\tau)}(u_\tau) < +\infty.  
\end{equation*}
Again, it suffices to show (\ref{eq: sufficient condition}). We set $v_\tau\in J_\tau \phi[u_\tau]$ and as in example \ref{ex: oscillation example} it holds that $v_\tau \to u \ (\tau\to 0)$. We have
\begin{eqnarray*}
&& \frac{\phi_{\epsilon(\tau)}(u_\tau) - \cY_\tau\phi_{\epsilon(\tau)}(u_\tau)}{\tau} \ + \ \frac{\cY_\tau\phi(u_\tau) - \phi(u_\tau)}{\tau} \\
&\geq& \frac{\phi_{\epsilon(\tau)}(u_\tau) - \phi_{\epsilon(\tau)}(v_\tau)}{\tau} + \frac{\phi(v_\tau) - \phi(u_\tau)}{\tau} \ = \ \frac{\epsilon(\tau)}{\tau} \left(\zeta_{\epsilon(\tau)}(u_\tau) - \zeta_{\epsilon(\tau)}(v_\tau)\right), 
\end{eqnarray*} 
and we can conclude. 
\end{proof}
\end{example}

\section{Perturbations, time-space discretizations}\label{sec: perturbations and co}

We return to the general case of a complete metric space $(\SS, d)$ with a (compatible) topology $\sigma$ on it, in accordance with the topological assumptions of section \ref{subsec:top ass}, and study some stimulating aspects. 

\subsection{Perturbations}\label{subsec: perturbations} 

Let a functional $\phi: \SS \rightarrow (-\infty, +\infty]$ satisfying assumption \ref{ass: AGS1} and functionals $\cP_\epsilon: \SS \rightarrow (-\infty, +\infty]$ $(\epsilon > 0)$ be given with  
\begin{equation}
\cP_\epsilon (\cdot) \geq -\tilde{A}-\tilde{B}d^2(\cdot, \tilde{u}_\star) 
\end{equation}
for some $\tilde{A}, \tilde{B} > 0, \ \tilde{u}_\star\in\SS$. We suppose that for $\epsilon_n \to 0 \ (\epsilon_n > 0)$ it holds that
\begin{equation}\label{eq: locally bounded}
\sup_{n,m} d(u_n, u_m) < +\infty, \ u_n\stackrel{\sigma}{\weakto}u \ \Rightarrow \ \sup_{n\in\N} |\cP_{\epsilon_n}(u_n)| < +\infty. 
\end{equation} 

We define 
\begin{equation}
\phi_\epsilon(u) = \phi(u) + \epsilon \cP_\epsilon(u). 
\end{equation}
The functionals $\phi_\epsilon: \SS\rightarrow (-\infty, +\infty]$ $(\epsilon > 0)$ satisfy assumption \ref{ass: co}, and moreover, assumption \ref{ass: liminf} holds. 
In order to apply Theorem \ref{thm: main theorem} it remains to check assumption \ref{ass: crucial assumption}: 

\begin{proposition}\label{prop: perturbations}
Let $\phi$ and $\phi_\epsilon = \phi + \epsilon \cP_\epsilon $ be given as above. If we choose $\epsilon = \epsilon(\tau)$ with $\frac{\epsilon(\tau)}{\tau} \to 0$ $(\tau\to0)$, then assumption \ref{ass: crucial assumption} is satisfied with this choice $\epsilon = \epsilon(\tau) \ (\tau > 0)$. In particular, all the results of Theorem \ref{thm: main theorem} hold.  
\end{proposition} 
  
\begin{proof}
We prove assumption \ref{ass: crucial assumption} with similar arguments as in example \ref{ex: finite dimensional perturbation}. Let $u, u_\tau\in \SS \ (\tau > 0)$ be given with 
\begin{equation*}
u_\tau \stackrel{\sigma}{\weakto} u, \ \sup_\tau \left\{\phi_{\epsilon(\tau)}(u_\tau), d(u_\tau, u)\right\} < +\infty. 
\end{equation*} 
We have already remarked several times that if a functional satisfies assumption \ref{ass: AGS1}, Proposition \ref{prop: important proposition} is applicable to it. Hence, in order to prove assumption \ref{ass: crucial assumption}, it suffices to show 
\begin{equation*}
\mathop{\liminf}_{\tau\to0} \frac{\phi_{\epsilon(\tau)}(u_\tau) - \cY_\tau \phi_{\epsilon(\tau)}(u_\tau)}{\tau} \ \geq \ \mathop{\liminf}_{\tau\to 0} \frac{\phi(u_\tau) - \cY_\tau\phi(u_\tau)}{\tau}. 
\end{equation*} 
We set $v_\tau\in J_\tau \phi[u_\tau]$. The sequence $\left(\frac{d(u_\tau, v_\tau)}{\sqrt{\tau}}\right)_{\tau > 0}$ is bounded (see the beginning of the proof of Proposition \ref{prop: important proposition}). We have   
\begin{eqnarray*}
&& \frac{\phi_{\epsilon(\tau)}(u_\tau) - \cY_\tau\phi_{\epsilon(\tau)}(u_\tau)}{\tau} \ + \ \frac{\cY_\tau\phi(u_\tau) - \phi(u_\tau)}{\tau} \\
&\geq& \frac{\phi_{\epsilon(\tau)}(u_\tau) - \phi_{\epsilon(\tau)}(v_\tau)}{\tau} + \frac{\phi(v_\tau) - \phi(u_\tau)}{\tau} \ = \ \frac{\epsilon(\tau)}{\tau} \left(\cP_{\epsilon(\tau)}(u_\tau) - \cP_{\epsilon(\tau)}(v_\tau)\right).  
\end{eqnarray*}
Applying (\ref{eq: locally bounded}) we conclude. 
\end{proof}

In particular, if $|\partial^- \phi|$ is a strong upper gradient for $\phi$, the relaxed minimizing movement scheme (\ref{eq: recovery sequence}) - (\ref{eq: relaxed form of minimi}) along $(\phi + \epsilon\cP_\epsilon)_{\epsilon > 0}$ leads to curves of maximal slope for $\phi$ with respect to $|\partial^- \phi|$. 

\paragraph{Remark on the choice $\epsilon = \epsilon(\tau)$} Note that other choices $\epsilon = \epsilon(\tau)$ (e.g. of order $\cO(\tau)$) might be possible besides, depending on the special features of the functionals $\cP_\epsilon$. Indeed, in order to prove assumption \ref{ass: crucial assumption}, the following condition is sufficient:

Whenever $(u_\tau)_{\tau > 0}, (v_\tau)_{\tau > 0} \subset \SS$ satisfy
\begin{eqnarray*}
u_\tau, v_\tau \stackrel{\sigma}{\weakto} u, \ \sup_\tau \left\{\frac{d(u_\tau, v_\tau)}{\sqrt{\tau}}, d(u_\tau, u)\right\} < +\infty,  
\end{eqnarray*}
then it holds that
\begin{equation*}
\mathop{\lim}_{\tau\to0} \frac{\epsilon(\tau)}{\tau} \left(\cP_{\epsilon(\tau)}(u_\tau) - \cP_{\epsilon(\tau)}(v_\tau)\right) = 0. 
\end{equation*}

\paragraph{Remark on condition (\ref{eq: locally bounded})} If the $\sigma$-topology coincides with the topology induced by the distance $d$, an equivalent formulation of (\ref{eq: locally bounded}) is given by
\begin{equation}
\mathop{\limsup}_{n\to+\infty} \sup_{v: d(v,u)\leq\rho_n} |\cP_{\epsilon_n}(v)| \ < \ +\infty
\end{equation} 
for every $u\in\SS$, $\rho_n\to0 \ (\rho_n > 0)$.\\

\subsection{Restriction to bounded subsets}\label{subsec: restriction to bounded subsets} 

Step by step, we have gained the impression that the deciding factor in our theory is the \textbf{local} behaviour of the sequence of functionals under consideration, whereas its global behaviour appears to be irrelevant. 

Indeed, let functionals $\phi, \phi_\epsilon: \SS \rightarrow (-\infty, +\infty]$ $(\epsilon > 0)$ be given satisfying the assumptions \ref{ass: co}, \ref{ass: liminf} and assumption \ref{ass: crucial assumption} with the choice $\epsilon = \epsilon(\tau)$, and let $(r(\epsilon))_{\epsilon > 0}$ be an arbitrary sequence  of positive real numbers $r(\epsilon)$ with $r(\epsilon)\uparrow +\infty$ $(\epsilon \to0)$. For $u_{\star\star}\in\SS$ arbitrary but fixed, we define 
\begin{equation}
\tilde{\phi}_\epsilon = \phi_\epsilon + \I_{\left\{d(u_{\star\star}, \cdot) \leq r(\epsilon)\right\}},
\end{equation}
in which we set $\I_{\VV_\epsilon} \equiv 0$ on $\VV_\epsilon$, $\I_{\VV_\epsilon} \equiv +\infty$ on $\SS \setminus \VV_\epsilon$, for $\VV_\epsilon = \left\{d(u_{\star\star}, \cdot) \leq r(\epsilon)\right\}$.\\ 

Then the functionals $\phi, \tilde{\phi}_\epsilon: \SS\rightarrow (-\infty, +\infty]$ $(\epsilon > 0)$ satisfy the assumptions \ref{ass: co}, \ref{ass: liminf} (this is obviously true), and moreover, we can show that assumption \ref{ass: crucial assumption} holds for $\phi, (\tilde{\phi}_\epsilon)_{\epsilon > 0}$ as well, with the same choice $\epsilon = \epsilon(\tau)$. 
\begin{proof}
Let $u, u_\tau\in \SS \ (\tau > 0)$ be given with 
\begin{equation*}
u_\tau \stackrel{\sigma}{\weakto} u, \ \sup_\tau \left\{\tilde{\phi}_{\epsilon(\tau)}(u_\tau), d(u_\tau, u)\right\} < +\infty. 
\end{equation*} 
In order to prove assumption \ref{ass: crucial assumption}, it suffices to show 
\begin{equation*}
\mathop{\liminf}_{\tau\to0} \frac{\tilde{\phi}_{\epsilon(\tau)}(u_\tau) - \cY_\tau \tilde{\phi}_{\epsilon(\tau)}(u_\tau)}{\tau} \ \geq \ \mathop{\liminf}_{\tau\to 0} \frac{\phi_{\epsilon(\tau)}(u_\tau) - \cY_\tau\phi_{\epsilon(\tau)}(u_\tau)}{\tau}. 
\end{equation*} 
We set $v_\tau\in \SS$ with
\begin{equation*}
\phi_{\epsilon(\tau)}(v_\tau) + \frac{1}{2\tau} d^2(v_\tau, u_\tau) \leq \cY_\tau \phi_{\epsilon(\tau)}(u_\tau) + \tau^2.
\end{equation*} 
It holds that $d(u_\tau, v_\tau)\to 0$ $(\tau\to0)$. This can be proved with similar arguments as at the beginning of the proof of Proposition \ref{prop: important proposition}. Thus, there exists $\eta > 0$ such that
\begin{equation*}
v_\tau\in \left\{d(u_{\star\star}, \cdot) \leq r(\epsilon(\tau))\right\}
\end{equation*}
for all $0 < \tau < \eta$. So we have
\begin{eqnarray*}
&& \frac{\tilde{\phi}_{\epsilon(\tau)}(u_\tau) - \cY_\tau \tilde{\phi}_{\epsilon(\tau)}(u_\tau)}{\tau} \ + \ \frac{\cY_\tau\phi_{\epsilon(\tau)}(u_\tau) - \phi_{\epsilon(\tau)}(u_\tau)}{\tau} \\
&=& \frac{\cY_\tau\phi_{\epsilon(\tau)}(u_\tau) - \cY_\tau \tilde{\phi}_{\epsilon(\tau)}(u_\tau)}{\tau} 
\ \geq \ \frac{\phi_{\epsilon(\tau)}(v_\tau) - \tilde{\phi}_{\epsilon(\tau)}(v_\tau)}{\tau}  -  \tau \ = \ -\tau
\end{eqnarray*}
for all $\tau\in(0, \eta)$. The proof is finished. 
\end{proof}

The preceding proof also reveals that whenever the convergence condition (\ref{eq: convergence condition}) is satisfied for $\phi, (\phi_\epsilon)_{\epsilon > 0}$, then it is satisfied for $\phi, (\tilde{\phi}_\epsilon)_{\epsilon > 0}$ too. \\

All in all, we obtain that we can always replace (\ref{eq: relaxed form of minimi}) by 
\begin{equation}
(\phi_{\epsilon(\tau)} + \I_{\VV_{\epsilon(\tau)}}) (u_\tau^n) + \frac{1}{2\tau} d^2(u_\tau^n, u_\tau^{n-1}) \ \leq \ \cY_\tau (\phi_{\epsilon(\tau)} + \I_{\VV_{\epsilon(\tau)}}) (u_\tau^{n-1}) + \gamma_\tau \tau
\end{equation} 
with $\VV_{\epsilon(\tau)} = \left\{d(u_{\star\star}, \cdot) \leq r(\epsilon(\tau))\right\}$, and still, all the results of section \ref{sec: main theorem} hold.

\subsection{Time-space discretizations along a single functional}\label{subsec: time space discretizations}  

Let a functional $\phi: \SS\rightarrow (-\infty, +\infty]$ satisfying assumption \ref{ass: AGS1} and subsets $\WW_\epsilon\subset\SS$ $(\epsilon > 0)$ be given. We define
\begin{equation}
\phi_\epsilon = \phi + \I_{\WW_\epsilon},
\end{equation}
in which we set $\I_{\WW_\epsilon} \equiv 0$ on $\WW_\epsilon$, $\I_{\WW_\epsilon} \equiv +\infty$ on $\SS \setminus \WW_\epsilon$.\\

The assumptions \ref{ass: co}, \ref{ass: liminf} are clearly satisfied in this case. We want to derive conditions on $(\WW_\epsilon)_{\epsilon > 0}$ and $(\epsilon(\tau))_{\tau > 0}$ such that assumption \ref{ass: crucial assumption} holds:\\ 

Let $u, u_\tau\in \SS \ (\tau > 0)$ be given with 
\begin{equation*}
u_\tau \stackrel{\sigma}{\weakto} u, \ \sup_\tau \left\{\phi_{\epsilon(\tau)}(u_\tau), d(u_\tau, u)\right\} < +\infty. 
\end{equation*} 
In view of Proposition \ref{prop: important proposition}, it would suffice to show  
\begin{equation*}
\mathop{\liminf}_{\tau\to0} \frac{\phi_{\epsilon(\tau)}(u_\tau) - \cY_\tau \phi_{\epsilon(\tau)}(u_\tau)}{\tau} \ \geq \ \mathop{\liminf}_{\tau\to 0} \frac{\phi(u_\tau) - \cY_\tau\phi(u_\tau)}{\tau}. 
\end{equation*} 
We set $v_\tau\in J_\tau\phi[u_\tau]$ and we have
\begin{eqnarray*}
&& \frac{\phi_{\epsilon(\tau)}(u_\tau) - \cY_\tau \phi_{\epsilon(\tau)}(u_\tau)}{\tau} \ + \ \frac{\cY_\tau \phi(u_\tau) - \phi(u_\tau)}{\tau} \ = \ \frac{\cY_\tau\phi(u_\tau) - \cY_\tau \phi_{\epsilon(\tau)}(u_\tau)}{\tau} \\
&\geq& \frac{\phi(v_\tau) - \phi(w_\tau)}{\tau} + \frac{1}{2\tau^2}\left[d^2(u_\tau, v_\tau) - d^2(u_\tau, w_\tau)\right]\\
&\geq& \frac{\phi(v_\tau) - \phi(w_\tau)}{\tau} + \frac{1}{2\tau^2}\left[-2d(u_\tau, v_\tau) d(v_\tau, w_\tau) - d^2(v_\tau, w_\tau)\right]
\end{eqnarray*} 
for all $w_\tau \in \WW_{\epsilon(\tau)}$.

If there existed $(w_\tau)_{\tau > 0} \subset \SS, \ w_\tau\in \WW_{\epsilon(\tau)} \ (\tau > 0)$, such that 
\begin{equation}\label{eq: last line}
\mathop{\liminf}_{\tau\to0} \left(\frac{\phi(v_\tau) - \phi(w_\tau)}{\tau} - \frac{d(u_\tau, v_\tau) d(v_\tau, w_\tau)}{\tau^2} - \frac{d^2(v_\tau, w_\tau)}{2\tau^2}\right) \geq 0,
\end{equation} 
assumption \ref{ass: crucial assumption} would have been proved for this choice $\epsilon = \epsilon(\tau)$.\\

The sequence $\left(\frac{d(u_\tau, v_\tau)}{\sqrt{\tau}}\right)_{\tau > 0}$ is bounded (see the beginning of the proof of Proposition \ref{prop: important proposition}). So, a sufficient condition in order to show (\ref{eq: last line}) would be
\begin{eqnarray}
\mathop{\liminf}_{\tau\to0} \frac{\phi(v_\tau) - \phi(w_\tau)}{\tau} \geq 0, \label{eqn: tsd1} \\
\frac{d(v_\tau, w_\tau)}{\sqrt{\tau^3}} \to 0 \ (\tau\to0) \label{eqn: tsd2}. 
\end{eqnarray}

\begin{example} 
We set $\SS = L^2(\Omega)$ with $\Omega\subset \RN$ and define the functional $\phi: L^2(\Omega) \rightarrow (-\infty, +\infty]$, $D(\phi) = H^1(\Omega)$, 
\begin{equation*}
\phi(u) = \frac{1}{2}\int_{\Omega}^{}{\left|\nabla u(x)\right|^2 \ dx} + \int_{\Omega}^{}{f(x, u(x)) \ dx}, \text{ if } u\in H^1(\Omega). 
\end{equation*}
Let us suppose that there exist $c_1, c_2 > 0$ such that $f: \Omega \times \R \rightarrow \R$ satisfies 
\begin{eqnarray*}
|f(x,w) - f(x,z)| &\leq& c_1 (1 + |z| + |w|)|z-w|, \\
f(x, z) &\geq& -c_2 (1 + |z|^2) 
\end{eqnarray*}
for all $x\in\Omega, w,z \in \RN$. 

Then the conditions (\ref{eqn: tsd1}), (\ref{eqn: tsd2}) hold for $w_\tau\in \WW_{\epsilon(\tau)}$ with
\begin{eqnarray}
\left\|\nabla w_\tau\right\|_{L^2(\Omega)} \leq \left\|\nabla v_\tau\right\|_{L^2(\Omega)}, \label{eqn: fem1} \\
\frac{\left\|v_\tau - w_\tau\right\|_{L^2(\Omega)}}{\sqrt{\tau^3}} \to 0 \ (\tau\to0). \label{eqn: fem2} 
\end{eqnarray}
The sequence $\left(\left\|\nabla v_\tau\right\|_{L^2(\Omega)}\right)_{\tau > 0}$ is bounded. 
Now, estimates like (\ref{eqn: fem1}), (\ref{eqn: fem2}) are well-known in the finite element theory (see e.g. \cite{Ciarlet78}). 
\end{example}

\section{Some final remarks}\label{sec: some final remarks}

We would like to mention possible generalizations of our theory which we have not considered so far for the sake of clear presentation.\\

The theory can be easily extended to $p$-curves of maximal slope (see \cite{AmbrosioGigliSavare05} for the definition of $p$-curves of maximal slope) with $p\in(1, +\infty)$, as well as to time discretizations with non-equidistant time steps. \\

Note that we do not use the special structure of the relaxed slope in the proof of Theorem \ref{thm: main theorem}. Section \ref{sec: main theorem} and the ideas of section \ref{sec: convergence of gradient flows} remain valid if we replace $|\partial^- \phi|$ by some upper gradient $g:\SS\rightarrow [0, +\infty]$.

One might consider in addition to the distance $d$ a sequence $(d_\epsilon)_{\epsilon > 0}$ of distances with 
\begin{equation}
u_n \stackrel{\sigma}{\weakto} u, \ v_n\stackrel{\sigma}{\weakto}v \ \Rightarrow \ \mathop{\liminf}_{n\to\infty} d_{\epsilon_n}(u_n, v_n) \geq d(u, v),  
\end{equation}
replace (\ref{eq: relaxed form of minimi}) by 
\begin{equation}
\phi_{\epsilon(\tau)}(u_\tau^n) + \frac{1}{2\tau} d_{\epsilon(\tau)}^2(u_\tau^n, u_\tau^{n-1}) \ \leq \ \overline{\cY}_\tau \phi_{\epsilon(\tau)}(u_\tau^{n-1}) + \gamma_\tau \tau,
\end{equation}
in which
\begin{equation*}
\overline{\cY}_\tau \phi_{\epsilon(\tau)} (u) := \inf_{v\in \SS} \left\{\phi_{\epsilon(\tau)}(v) + \frac{1}{2\tau} d_{\epsilon(\tau)}^2(v,u) \right\}, \ u\in \SS, 
\end{equation*}
and appropriately adapt the assumptions of section \ref{sec: main theorem}. Then the proof of Theorem \ref{thm: main theorem} can be adapted as well in this case, with a yet refined version of Ascoli-Arzel\`a Theorem. The ideas of section \ref{sec: convergence of gradient flows} remain valid too. 

However, we do not want to expound the possible generalizations of our theory with regard to the topology, just give an impetus in case different topological assumptions are of interest.\\

\paragraph{Acknowledgement} I would like to thank Giuseppe Savar\'e, Martin Brokate and Daniel Matthes for stimulating discussions on this topic. I gratefully acknowledge PRIN10/11 grant ``Calculus of Variations'' from MIUR and the TopMath program from Technische Universit\"at M\"unchen for supporting my visit to the University of Pavia.

\bibliographystyle{siam}
\bibliography{bibliografia2015}
\end{document}